\documentclass[oneside,english,12pt]{amsart}
 \usepackage[T1]{fontenc}
 \usepackage[latin9]{inputenc}
 \usepackage{amsmath, amstext, amsthm, amsfonts, amssymb}
 \usepackage[all]{xy}
 \usepackage{graphicx}
 \usepackage{esint}
 \usepackage{units}
\usepackage[mathscr]{eucal}
 \usepackage{hyperref}
 \usepackage{amscd}
\usepackage{amsthm}
\usepackage{lmodern}
\usepackage{helvet}

\usepackage{setspace}
\numberwithin{equation}{section}
\usepackage{geometry}
\usepackage[all]{xy}
\xyoption{curve}
\usepackage{babel}
\providecommand{\tabularnewline}{\\}

\makeatletter

\def\k{\mathrm{K}}
\def\x{\mathrm{X}}
\def\Zr{\mathrm{Z}}
\def\xab{\x_{a,b}}
\def\kc{\check{\k}}
\def\kab{\k_{\alpha,\beta}}
\def\kabc{\kc_{\alpha,\beta}}
\def\E{\mathcal{E}}
\def\Ec{\check{\E}}

\def\a{\alpha}
\def\b{\beta}

\def\Dx{\Delta(\xi)}
\def\c{\mathcal{C}}
\def\z{\mathfrak{z}}
\def\zt{\tilde{\z}}
\def\Zab{\mathrm{Z}_{\a,\b}}
 
 \newcommand {\C} {{\mathbb C}}
 \newcommand {\R} {{\mathbb R}}
 \newcommand {\Z} {{\mathbb Z}}
 \newcommand {\Q} {{\mathbb Q}}

 \newcommand {\PP} {{\mathbb P}}
 
 \newcommand {\HH} {{\mathbb H}}

 \newcommand {\W} {{\mathcal W}}
 \newcommand {\D} {{\mathcal D}}
 
 \newcommand {\Ss} {{\mathcal S}}
 \newcommand {\Y} {{\mathcal Y}}

 \newcommand {\LL} {{\mathcal L}}

  \newcommand {\del} {{\partial}}
 \newcommand {\X} {{\mathscr X}}

 \newcommand {\CO}{\mathcal{O}}

 \newcommand {\cl}{\text{\rm cl}}

\newcommand{\alg}{\text{\rm alg}}

\newcommand{\Ch}{\text{\rm CH}}

\newcommand{\ol}{\overline}
\newcommand{\ul}{\underline}

\newcommand{\Ext}{\text{\rm Ext}}
\newcommand{\MHS}{\text{\rm MHS}}

\newcommand{\w}{\omega}
\newcommand{\tensor}{\otimes}
\newcommand{\isom}{\simeq}

\newtheorem{thm}{Theorem}[section]
\newtheorem{cor}{Corollary}[section]
\newtheorem{lemma}{Lemma}[section]
\newtheorem{prop}{Proposition}[section]

\newtheorem{rem}{Remark}[section]

\newtheorem{q}{Question}[section]

\newtheorem{assum}{Assumption}[section]
\newenvironment{lyxlist}[1]
{\begin{list}{}
{\settowidth{\labelwidth}{#1}
 \setlength{\leftmargin}{\labelwidth}
 \addtolength{\leftmargin}{\labelsep}
 }}
{\end{list}}

\newcommand{\thmref}[1]{Theorem \ref{#1}}
\newcommand{\propref}[1]{Proposition \ref{#1}}

\newcommand{\coref}[1]{Corollary \ref{#1}}

\DeclareMathOperator{\CH}{CH}
\DeclareMathOperator{\Span}{Span}

\makeatother

\begin{document}

\title[Normal functions]{Normal functions, Picard-Fuchs equations, and elliptic fibrations on K3 surfaces}

\author{Xi Chen${}^{\dagger}$}
\address{632 Central Academic Building\\ University of Alberta\\ Edmonton, Alberta T6G 2G1, CANADA}
\email{xichen@math.ualberta.ca}

\author{Charles Doran${}^{\dagger}$}
\address{632 Central Academic Building\\ University of Alberta\\ Edmonton, Alberta T6G 2G1, CANADA}
\email{doran@math.ualberta.ca}

\author{Matt Kerr${}^{\dagger\dagger}$}
\address{Department of Mathematics, Campus Box 1146\\ Washington University in St. Louis\\ St. Louis, MO 63130, USA}
\email{matkerr@math.wustl.edu}

\author{James D. Lewis${}^{\dagger}$}
\address{632 Central Academic Building\\ University of Alberta\\ Edmonton, Alberta T6G 2G1, CANADA}
\email{lewisjd@ualberta.ca}

\date{\today}

\begin{abstract}
Using Gauss-Manin derivatives of generalized normal functions,
we arrive at some remarkable results on the non-triviality of the 
 transcendental regulator for $K_m$ of a very general  projective algebraic manifold.
Our strongest results are for the transcendental regulator
for $K_1$ of a very general $K3$ surface.  We also construct an explicit 
family of $K_1$ cycles on $H \oplus E_8 \oplus E_8$-polarized $K3$ surfaces,
and show they are indecomposable by a direct evaluation of the real regulator.  
Critical use is made of natural elliptic fibrations, hypersurface normal forms, 
and an explicit parametrization by modular functions.
\end{abstract}

%\thanks{${}^{\dagger}$ Partially supported by a grant from the Natural
%Sciences and Engineering Research Council of Canada.
%${}^{\dagger\dagger}$ Partially supported by NSF grant DMS-1068974}



\maketitle

\tableofcontents{}

\section{Introduction}\label{SEC01} 
The subject of this paper is the existence, construction, and detection of indecomposable
algebraic $K_1$-cycle classes on $K3$ surfaces and their self-products.  We begin by treating the
existence of regulator indecomposables on a very general $K3$ with fixed polarization by a lattice
of rank less than $20$ (\S 2), as well as on their self-products in the rank one projective case (\S 4).
This is intertwined with a discussion (\S 3) of homogeneous and inhomogeneous Picard-Fuchs equations
for truncated normal functions --- a subject of increasing interest due to their recent spectacular use
in open string mirror symmetry \cite{MW} --- which is further amplified by explicit examples in \S 5.

The second half of the paper takes up the question of how to use the geometry of polarized $K3$ surfaces
with high Picard rank to construct indecomposable cycles (\S \S 5-6).  Elliptic fibrations yield an extremely natural
source of families of cycles, whose image under the real and transcendental regulator maps have apparently
not been previously studied.  Our computation of their real regulator not only proves indecomposability, but turns out
to be related to higher Green's functions on the modular curve $X(2)$ (cf. \cite{Ke}).  The paper concludes (\S 7) with a discussion of the mysterious Picard rank 20 case and its relationship to open irrationality problems.  In the remainder of this introduction, we shall state the main existence results of \S\S 2-4, and place the constructions of \S6 in historical context.

Let $X$ be a projective algebraic manifold of dimension $d$,
and $\Ch^r(X,m)$ the higher Chow group introduced by Bloch (\cite{B}). We are mainly interested
in working modulo torsion, thus we will restrict ourselves to the corresponding group $\Ch^r(X,m;\Q) := \Ch^r(X,m)\tensor \Q$. An explicit description of the Bloch cycle class map to Deligne cohomology,
\[
\cl_{r,m} :\Ch_{\hom}^r(X,m;\Q) \to J\big(H^{2r-m-1}(X,\Q(r))\big) \subset H_{\D}^{2r-m}(X,\Q(r)),
\]
is given in \cite{KLM}, where
\[
\begin{split}
&J\big(H^{2r-m-1}(X,\Q(r))\big) := \Ext_{\MHS}^1\big(\Q(0),H^{2r-m-1}(X,\Q(r))\big)\\
&\quad \quad \quad
\simeq \frac{F^{d-r+1}H^{2d-2r+m+1}(X,\C)^{\vee}}{H_{2d-2r+m+1}(X,\Q(d-r))}.
\end{split}
\]
 We will now assume that $X$ is a very general member of a  family
 $\lambda: \X \to \Ss$, where $\X$, $\Ss$ are smooth quasi-projective varieties
 and $\lambda$ is smooth and proper, and where $X := \lambda^{-1}(0)$ corresponds to
 $0\in \Ss$. Associated to this is the Kodaira-Spencer map
 $\kappa : T_0(\Ss) \to H^1(X,\Theta_X)$, whose image we will denote by
 $H_{\alg}^1(X,\Theta_X)$, where $\Theta_X$ is the sheaf of holomorphic vector fields
 on $X$.

 \bigskip
 The Hodge structure
 \[
 H^{2r-m-1}(X,\Q(r)) = H_f^{2r-m-1}(X,\Q(r)) \bigoplus
 H_v^{2r-m-1}(X,\Q(r))
 \]
  decomposes, where
 \[
 H_f^{2r-m-1}(X,\Q(r)) := H^{2r-m-1}(X,\Q(r))^{\pi_1(\Ss)}
 \]
 is the fixed part of the corresponding
 monodromy group action  on $H^{2r-m-1}(X,\Q(r))$, and $H_v^{2r-m-1}(X,\Q(r))$
 is the orthogonal complement.  
  
 \medskip
 One has a reduced cycle class map
 \[
\ul{\cl}_{r,m}: \Ch^r(X,m;\Q) \to J\big(H^{2r-m-1}_v(X,\Q(r))\big).
 \]
Such a regulator plays a key role in detecting
 interesting $\Ch^r(X,m)$ classes, such as indecomposables (see for example \cite{L}
 and \cite{MS}). We can further pass to the transcendental regulator
 \[
 \Phi_{r,m} : \Ch^r(X,m;\Q)\to \frac{F^{d-r+m+1}H_v^{2d-2r+m+1}(X,\C)^{\vee}}{H^v_{2d-2r+m+1}(X,\Q(d-r))},
 \]
for which the formula is given as follows. 
If $\xi \in \Ch^r(X,m;\Q)$, then viewing $|\xi|\subset X\times \C^m$ as a closed subset of
 codimension $r$, it follows that $\dim Pr_X(|\xi|) \leq d-r+m$. According the
 the formula in \cite{KLM} (also cf. \cite{K-L}), one can choose $\xi$ such that
 it meets the real cube $X\times [-\infty,0]^m$ properly, and that for $\w\in
 F^{d-r+m+1}H_v^{2d-2r+m+1}(X,\C)$,
 \[
 \Phi_{r,m}(\xi)(\w) = \pm \frac{1}{(2\pi{\rm i})^{d-r}}\int_{\del^{-1}\{\xi \cap X\times [-\infty,0]^m\}}\w.
 \]
For example in the case $(d,r,m) = (2,2,1)$, we have
\[
\Phi_{2,1} : \Ch^2(X,1)\to \frac{H_v^{2,0}(X)^{\vee}}{H^v_2(X,\Q)}.
 \]
 Our first result is the following.
 \begin{thm}\label{T27} Let $X/\C$ be a very general algebraic $K3$ surface. Then
the transcendental regulator $\Phi_{2,1}$
is non-trivial. Quite generally, if $X$ is a very general member of a general subfamily
 of dimension $20-\ell$, describing a family of $K3$ surfaces with general member of
 Picard rank $\ell$,
 with $\ell < 20$, then  $\Phi_{2,1}$ is non-trivial.
 \end{thm}
 
 From the proof of Theorem \ref{T27}, we deduce:
 
  \begin{cor}\label{CO78} Let $X/\C$ be a very general member of a family of surfaces
 for which $H_{\alg}^1(X,\Theta_X)\otimes H_v^{2,0}(X) \to H_{v}^{1,1}(X)$ is surjective.
 If the real regulator $r_{2,1} : \Ch^2(X,1) \to H_{v}^{1,1}(X,\R(1))$ is non-trivial,
 then so is the transcendental regulator $\Phi_{2,1}$.
 \end{cor}
 
 Now consider  $X$ of dimension $d$ as a very general member of a  family
 $\lambda: \X \to \Ss$.
 With a little bit of effort, one can also show
 the following.

\begin{thm}\label{T28} Suppose that the cup product induced map
\[
H^1_{\alg}(X,\Theta_X)\otimes H_v^{d-r+m+1-\ell,d-r+\ell}(X) \to H_v^{d-r+m-\ell,d-r+\ell+1}(X),
\]
is surjective for all $\ell = 0,...,m-1$.
Then $\ul{\cl}_{r,m}\ne 0 \Rightarrow \Phi_{r,m}\ne 0$.
\end{thm}

Theorems \ref{T27}, \ref{T28} and Corollary \ref{CO78} will be proved in section \ref{SEC02}. We deduce from Theorem \ref{T28} the following:

\begin{cor}\label{C27} Let $X$ be a very general $K3$ surface, and $H_v^2(X,\C)$ be
transcendental cohomology.  Then  the transcendental regulator
 \[
 \Phi_{3,1} : \Ch^3(X\times X,1) \to \frac{\big\{F^3\big(H^2_v(X,\C)\otimes H^2_v(X,\C)
 \big)\big\}^{\vee}}{H_4(X\times X,\Q(1))},
 \]
 is non-trivial.
\end{cor}

We prove Corollary \ref{C27} in section \ref{SEC03}. In turns out however, that with more effort, we can
actually prove the following stronger result:

\begin{thm}\label{T53}
The purely transcendental regulator
 \[
\Psi_{3,1} : \Ch^3(X\times X,1) \to \frac{H^{4,0}(X\times X,\C)^{\vee}}{H_4(X\times X,\Q(1))},
 \]
 is non-trivial for a very general $K3$ surface $X$.
\end{thm}

The proofs of all the above results rely on a very simple trick involving the
infinitesimal invariant of a normal function associated to a family of cycles
on $\X/\Ss$ inducing a given transcendental regulator value on $X$. A 
deeper question asks whether such a normal function is detected by
a Picard-Fuchs operator. A blanket answer to this question is a yes;
however rather than explain it here, we provide a complete clarification in \S \ref{SEC022}.
\medskip

Now returning to Theorem \ref{T27}, two questions come to mind. First, the
method of {\cite{C-L1}}, which proves the existence of deformations of
decomposables on Picard rank 20 $K3$'s, to indecomposables on a general
polarized $K3$, is highly non-explicit. How can one construct interesting
\emph{explicit} examples of cycles with nontrivial $\Phi_{2.1}$ on
subfamilies with $\ell>1$? Second, on a Picard-rank 20 $K3$, does
one \emph{expect} there to be any cycles at all which have nontrivial
$\Phi_{2,1}$, and which are therefore indecomposable?

The first question is our main concern for the remainder of the paper.  In \S \ref{SECImz}, 
we introduce a crucial set of tools needed for explicit computations in this setting.  
The notion of a polarized K3 surface is extended to that of a lattice polarization, and 
algebraic hypersurface normal forms are given for certain families of lattice polarized
K3 surfaces of high Picard rank $\ell$.  We then describe a very useful ``internal structure'' consisting of
an elliptic fibration with section(s).  Explicit Picard-Fuchs operators are given and related to
parametrizations of coarse moduli spaces by modular functions and their generalizations.

Starting in \S \ref{SEC04}, we restrict our considerations to $\ell=18$ or $19$,
where there have been a number of ideas that have not panned out.
The article by {\cite{PL-MS}}, which in itself is an interesting piece of work, 
considers a cycle $\mathcal{Z}$ on a 1-parameter
family of elliptically fibered $K3$'s with $\ell=19$ and a choice
of section $\omega$ of the relative canonical bundle. In this context
$F:=\Phi_{2,1}(\mathcal{Z})(\omega)$ is a multivalued holomorphic
function and the indecomposability of $\mathcal{Z}$ may be detected
by showing the Picard-Fuchs operator for $\omega$ does not annihilate
it. Unfortunately, this cycle turns out to be 2-torsion,\footnote{the cycle, which is supported over $\{Z=0\}\cup\{Z=1\}\cup\{X=0\}\cup\{X=\infty\}$
in the notation of {\cite{PL-MS}}, is in fact one-half the residue of
the symbol $\left\{ X,1-\frac{1}{Z}\right\} $.
} and the computation of $F$ leaves out a part of the membrane intergral
which cancels the part written down. For $\ell=18$, one can try to
construct regulator-indecomposable cycles on a product $E_{1}\times E_{2}$
of elliptic curves and then pass to the Kummer. Such a construction
is attempted in {\cite{G-L}} but this cycle, too, was shown by M. Saito
to be decomposable.\footnote{that construction can, however, be corrected \cite{T}.
} When $E_{1}\cong E_{2}$, other authors (cf. {\cite{Zi}}) have investigated
{}``triangle cycles'' supported on $E\times\{p\}$, $\{q\}\times E$,
and the diagonal $\Delta_{E}$, where $[p]-[q]$ is $N$-torsion.
But this cannot produce indecomposable cycles, since the sum of the
natural $N^{2}$ $N$-torsion translates of such a cycle (by integer
multiples of $p-q$ on the two factors) is both visibly decomposable
and (up to torsion) equivalent to $N^{2}$ times the original cycle.

With this discouraging history, it is easy to imagine that when $X$
is an elliptically fibered $K3$, the very natural $\CH^{2}(X,1)$
classes supported on semistable singular (Kodaira type $I_{n}$) fibers
might be decomposable as well. Indeed one knows in the case of a modular
elliptic fibration ($K3$ or not), that Beilinson's Eisenstein symbols
{\cite{Be}} kill all such classes. On the other hand, using arithmetic
methods to bound the rank of the $\text{dlog}$ image, Asakura {\cite{As}}
demonstrated that for elliptic surfaces with general fiber $y^{2}=x^{3}+x^{2}+t^{n}$
($n\in[7,29]$ prime), the type $I_{1}$ fibers generate $n-1$ independent
indecomposable $K_{1}$ classes. His paper stops short of attempting
any regulator computations for such cycles, and this is what we take
up in $\S3$ in the context where the surface and cycle are allowed
to vary.

Specifically, using an $I_{1}$ fiber in an internal elliptic fibration
of the 2-parameter family $\{\xab\}$ of Shioda-Inose $K3$'s ($\ell=18$)
{\cite{C-D2}}, we write down a (multivalued) family of cycles $\mathcal{Z}_{a,b}\in \CH^{2}(\xab,1)$.
Passing to the associated Kummer family with parameters $\alpha,\beta$
(and cycle $\Zab$), we find that the family of cycles becomes single-valued
over the diagonal ($\ell=19$) sublocus $\a=\b$, which is the Legendre
modular curve $\PP^{1}\backslash\{0,1,\infty\}\cong\mathfrak{H}/\Gamma(2)$.
At this point we write down a smooth family of real closed $(1,1)$
forms $\eta_{\a}$ and compute directly the function \[\begin{matrix} \psi(\alpha):=r_{2,1}(\Zr_{\a,\a})(\eta_{\a})= 
\\ \\
-8|\a+1|{\rm Im}\int_{\C}z\cdot\log\left|\frac{z+i}{z-i}\right|\left\{ \begin{array}{c}
\frac{\left\{ (\a^{2}-\a-1)z^{4}+2z^{2}+(\a^{3}-\a^{2}-2\a+1)\right\} }{|z^{2}-\a||1-\a z^{2}||z^{2}+1||z^{2}-(1+\a-\a^{2})|}\times\\ \\ 
\frac{\overline{\left\{ (\a^{3}-\a^{2}-2\a+1)z^{4}+2z^{2}+(\a^{2}-\a-1)\right\} }}{|(1+\a-\a^{2})z^{2}-1||z^{4}+(\a^{3}-3\a)z^{2}+1|}\end{array}\right\} dx\wedge dy
\end{matrix}
\] to be nonzero. By Corollary 1.3 we have immediately the
\begin{thm}
$\Phi_{2,1}(\mathcal{Z}_{a,b})$ is non-trivial for very general $(a,b)$,
and $\mathcal{Z}_{a,b}$ is indecomposable.\end{thm}

In light of the past confusion surrounding such constructions,
such a natural source of indecomposable cycles seems to us an
important development.  While the explicit formula above may
not look promising, $\psi(\alpha)$ is in fact a very interesting function. Dividing
out by the volume of the Legendre elliptic curve and pulling back
by the classical modular function $\lambda$ to obtain a function
$\tilde{\psi}(\tau)$ on $\mathfrak{H}$, yields a {}``Maass cusp
form with two poles''. That is, $\tilde{\psi}$ is $\Gamma(2)$-invariant,
is smooth away from the $\lambda$-preimage of $\a=\{-1,2\}$ (where it
has $\log|\cdot|$ singularities), dies at the 3 cusps, and (away
from these bad points) is an eigenfunction of the hyperbolic Laplacian
$-y^{2}\Delta$. This will be shown in a follow-up paper of the third
author {\cite{Ke}}.\footnote{which, it should be noted, relies crucially
on the computation here.}

Finally, we turn briefly to the second question, concerning the case
$\ell=20$, in $\S\ref{SECTrReg}$. Due to the vanishing of $H_{v}^{1,1}(X,\R)$,
$r_{2,1}$ is zero by definition, but this is no reason for the transcendental
Abel Jacobi map $\Phi_{2,1}$ to vanish. In the example we work out,
whether or not $\Phi_{2,1}(\mathcal{Z})$ is nontorsion boils down
to the irrationality of a single number (cf. \eqref{kappa}), which
we do not know how to prove directly. It seems likely both that the
cycle is indecomposable and that this may be shown by using the methods
in {\cite{As}} to compute the $\text{dlog}$ image.

\subsection*{Acknowledgments}

X. Chen, C. Doran and J. Lewis are partially supported by grants from the Natural Sciences and Engineering
Research Council of Canada.  M. Kerr is partially supported by National Science Foundation grant DMS-1068974.  The authors also thank Adrian Clingher for helpful conversations.

\section{Derivatives of normal functions I}\label{SEC02}

\subsection{Gauss-Manin derivatives}

We first prove Theorem \ref{T27} and Corollary \ref{CO78},  after which the general argument
pertaining to Theorem \ref{T28} will follow rather easily.
Consider a smooth family $\pi: \X\to S$ of $K3$ surfaces 
polarized by a relatively ample line bundle $L$
over a polydisk $S$, with central fiber $X$. We have the Gauss-Manin (GM) connection:
\begin{equation}\label{E900}
\nabla: \CO_{S}\otimes R^q\pi_* \C \to \Omega^1_{S} \tensor R^q \pi_* \C
\end{equation}
which is a flat connection that $\nabla^2 = 0$ and
satisfies the Griffiths transversality:
\[
\nabla  \big(\CO_{S}\otimes F^pR^q\pi_* \C\big) \subset
\Omega^1_{S}\tensor  F^{p-1}R^q \pi_* \C.
\]

Let $\Theta_S$ be the holomorphic tangent bundle of $S$. We can think of $\Theta_S$ as the sheaf of holomorphic
linear differential operators. By identifying $\partial/\partial z_k$ with
$\nabla_{\partial/\partial z_k}$, $\Theta_S$ acts on $\CO_S\tensor R^q\pi_* \C$ via
\begin{equation}\label{E907}
u \cdot\omega = \nabla_{u}\omega
\end{equation}
for $u\in H^0(\Theta_S)$, where we write $H^0(-)$ for $H^0(S,-)$.
\medskip

We fix a nonzero section $\omega\in H^0(K_{\X/S})$.
For all $u\in H^0(\Theta_S)$ and  all $\gamma\in H^2(X, \C)$
(where  $H^2(X, \C)$ is identified with $ H^0(S,R^2\pi_{\ast}\C)$,
using $S$ a polydisk),
\begin{equation}\label{E901}
u\langle \gamma,\w\rangle = \langle \gamma,\nabla_u\w\rangle.
\end{equation}
Let $\xi\in \CH^2(\X/S, 1)$ be the be the result of an algebraic deformation 
of a cycle in the central fiber $X$ restricted to $\X/S$, and $\cl_{2,1}$ be the regulator map
\begin{equation}\label{E903}
\cl_{2,1}: \CH^2(\X/S, 1) \to H^0\biggl(\frac{\CO_S\tensor R^2 \pi_* \C}{\CO_S\tensor F^2
R^2\pi_* \C + R^2\pi_* \Q(2)}\biggr).
\end{equation}
We let $\nu$ be a lift of $\cl_{2,1}(\xi)$ to $H^0(\CO_S\tensor R^2\pi_* \C)$.
We know that $\langle \nabla_u \nu, \omega\rangle = 0$ since the map
\begin{equation}\label{E910}
\begin{split}
\nabla\circ \cl_{2,1}: \CH^2(\X/S, 1) &\xrightarrow{\cl_{2,1}}
H^0\biggl(\frac{\CO_S\tensor R^2\pi_* \C}{\CO_S\tensor
F^2 R^2 \pi_* \C + R^2\pi_* \Q(2)}\biggr)\\
& \xrightarrow{\nabla}
H^0\biggl(\Omega^1_S\tensor \frac{R^2\pi_* \C}{F^1 R^2\pi_* \C} \biggr)
\end{split}
\end{equation}
induced by the GM connection is trivial.
This follows from the horizontality condition
on normal functions associated to (higher Chow) algebraic cycles - well known
among experts, the horizontality
condition (leading to the infinitesimal invariant of normal functions) being
deducible for example from a Deligne cohomology spectral sequence argument
in \cite{C-MS-P} (pp. 267-269), and adapted to higher Chow cycles.
 [For the convenience of the reader, here is how the argument
works. Recall the analytic Deligne complex $0\to \Z(r) \to \Omega_{\X}^{\bullet<r}$, which leads to
an exact sequence $\HH^{2r-m-1}(\Omega_{\X}^{\bullet<r}) \to H^{2r-m}_{\D}(\X,\Z(r)) \to
H^{2r-m}(\X,\Z(r))$. We consider a null-homologous cycle in $\Ch^r(X,m)$ 
that spreads to a (relatively null-homologous) cycle on $\Ch^r(\X/S,m)$, which
will map to zero in $H^{2r-m}(\X,\Z(r))$ (as $S$ is a polydisk), and hence the induced
normal function has a lift in $\HH^{2r-m-1}(\Omega_{\X}^{\bullet<r})$.  The Leray spectral sequence
gives us an edge map $\HH^{2r-m-1}(\Omega_{\X}^{\bullet<r}) \to H^0(S,
\R^{2r-m-1}\pi_{\ast} \Omega_{\X}^{\bullet<r})$. One has a filtering of the complex
$\LL^{\nu}\Omega^{\bullet<r}_{\X} :=$ Image $\big(\pi^{\ast}\Omega_{S}^{\nu}\otimes \Omega^{\bullet<r-\nu}_{\X} \to \Omega_{\X}^{\bullet<r}\big)$, with $Gr_{\LL}^{\nu} = \pi^{\ast}\Omega_S^{\nu}\otimes
\Omega_{X/S}^{\bullet < r-\nu} \simeq \Omega_S^{\nu}\otimes \Omega_{X/S}^{\bullet < r-\nu}$.
There is a spectral sequence computing $\R^{p+q}\pi_{\ast}\Omega_{\X}^{\bullet<r}$ with
$\E_1^{p,q} = \R^{p+q}Gr^p_{\LL} = \Omega_S^p\otimes\R^q
\rho_{\ast}\Omega^{\bullet <r-p}_{\X/S}$. So we have the composite
$H^0(S,\R^{2r-m-1}\pi_{\ast}\Omega_X^{\bullet<r}) \to H^0(S,\E_1^{0,2r-m-1})
\xrightarrow{d_1} H^0(S,\E_1^{1,2r-m-1})$, which must be zero by spectral sequence
degeneration, using the fact that $\E^{0,2r-m-1}_{\infty} \subset \ker \big(d_1 :
\E^{0,2r-m-1}_1 \to \E^{1,2r-m-1}_1\big)$.  But $H^0(S,\E_1^{0,2r-m-1})
\xrightarrow{d_1} H^0(S,\E_1^{1,2r-m-1})$ is precisely  the Gauss-Manin connection
\[
H^0(S,\R^{2r-m-1}\pi_{\ast}\Omega_{\X/S}^{\bullet<r}) \xrightarrow{\nabla}
H^0(S,\Omega_S^1\otimes\R^{2r-m-1}\pi_{\ast}\Omega_{\X/S}^{\bullet<r-1}).]
\]

\subsection{Nontriviality of transcendental regulators}

Now assume to the contrary that $\cl_{2,1}(\xi)(\w)$
is trivial. Then $\langle \nu,\w\rangle$ is a period, i.e. $\langle \nu,\w\rangle 
= \langle \gamma,\w\rangle$ for some $\gamma\in H^2(X,\Q(2))$. Applying $\nabla_{u}$
together with the horizontality condition on normal functions, we deduce that
\begin{equation}\label{E007}
\langle \gamma,\nabla_u\w\rangle = u\langle \gamma,\w\rangle = 
u\langle \nu,\w\rangle = \langle \nabla_u\nu,\w\rangle + \langle \nu,\nabla_u\w\rangle
=  \langle \nu,\nabla_u\w\rangle.
\end{equation}

It is well known that the projection of $\nabla_u \omega$ to $H^{1,1}(X_t)$ is the cup product
of $\kappa(u)$ and $\omega$, where $\kappa$ is the Kodaira-Spencer map
\begin{equation}\label{E912}
\Theta_{S,t} \xrightarrow{\kappa} H^1(X_t, \Theta_{X_t})
\end{equation}
at a point $t\in S$ (Griffiths). The following proposition, which is likely well-known,
shows that this cup product is surjective for $K3$
surfaces.

\begin{prop}\label{PROP950}
For a $K3$ surface $X$, the map
\begin{equation}\label{E950}
H^1(X, \Theta_X) \tensor H^{2,0}(X) \to H^{1,1}(X)
\end{equation}
induced by the contraction $\Theta_X \tensor \wedge^2 \Omega^1_X\to \Omega^1_X$
is an isomorphism, where $\Theta_X$ is the tangent bundle of $X$.
\end{prop}

\begin{proof} It is instructive to provide a simple proof of this fact.
The map \eqref{E950} gives rise to a pairing
\begin{equation}\label{E951}
H^1(X, \Theta_X) \tensor H^{1,1}(X)^\vee \to H^{2,0}(X)^\vee.
\end{equation}
Then \eqref{E950} is an isomorphism if and only if \eqref{E951} is a nondegenerate pairing.
Combining with Kodaira-Serre duality
\begin{equation}\label{E952}
H^{1,1}(X)^\vee = H^{1,1}(X) \text{ and } H^{2,0}(X)^\vee \isom H^{0,2}(X),
\end{equation}
we see that this pairing becomes
\begin{equation}\label{E953}
H^1(X,\Theta_X)\tensor H^1(X,\Omega^1_X) \to H^2(X,\CO_X)
\end{equation}
which is induced by the nature map $\Theta_X\tensor \Omega^1_X \to \CO_X$.
Therefore,
we have the commutative diagram
\begin{equation}\label{E954}
\xymatrix{
H^1(X,\Theta_X)\tensor H^1(X,\Omega^1_X) \ar[r] \ar[d]^{\tensor \omega} &
H^2(X,\CO_X) \ar[d]^{\tensor \omega}\\
H^1(X,\Theta_X)\tensor H^1(X,\Omega^1_X\tensor K_X) \ar[r] & H^2(X,K_X)
}
\end{equation}
for all $\omega\in H^0(X,K_X)$. The bottom row of \eqref{E954} is Serre duality and is hence a nondegenerate
pairing. Then the nondegeneracy of the top row follows easily when $K_X = \CO_X$.
\end{proof}

Note that $H^1(X,\Theta_{X})$ corresponds to all deformations (including
non-algebraic) of $X$. Let $H_{\alg}^1(X,\Theta_{X})$ correspond to
the algebraic deformations. For a general polarized $K3$ surface $(X,L)$,
$H_{\alg}^1(X,\Theta_{X})$ is the subspace $[c_1(L)]^\perp$.

\begin{cor}\label{C28} For a $K3$ surface $X$, the map
\[
H_{\alg}^1(X, \Theta_X) \tensor H^{2,0}(X) \to H_v^{1,1}(X),
\]
is an isomorphism.
\end{cor}

Suppose that the family $\pi: \X\to S$ is maximum, i.e., the image of the 
Kodaira-Spencer map $\kappa$ is
$[c_1(L)]^\perp$ at each point $t\in S$. Then by \coref{C28}, the projections of $\nabla_u \omega$ to
$\CO_S\otimes R^2 \pi_* \Omega_{\X/S}$, together
with $\w$, generate the subbundle $[c_1(L)]^\perp \cap \CO_S\otimes F^1R^2 \pi_* \C$ as $u$ varies in $H^0(\Theta_S)$. 
By (\ref{E007}), this
cannot happen if the reduced regulator $\ul{\cl}_{2,1}(\xi)$ is non-trivial, which was proven
in \cite{C-L1}. Finally,  we use the fact that $T_0(S) \simeq H^1_{\alg}(X,\Theta_X)$ together
with \cite{C-L1} to deduce
the latter statement in Theorem \ref{T27}. Corollary \ref{CO78} follows accordingly.

\begin{proof} (of Theorem \ref{T28}) Let us assume that $\Phi_{r,m}$ is zero. That means
that $\cl_{r,m}(\xi)$ is a period with respect to (acting on forms in) 
$F^{d-r+m+1}H_v^{2d-2r+m+1}(X,\C)$.
Then from the surjection of
\[
H^1_{\alg}(X,\Theta_X)\otimes H_v^{d-r+m+1-\ell,d-r+\ell}(X) \to H_v^{d-r+m-\ell,d-r+\ell+1}(X),
\]
in the case $\ell=0$, we deduce likewise that $\cl_{r,m}(\xi)$ is a period 
with respect to $F^{d-r+m}H_v^{2d-2r+m+1}(X,\C)$.
By iterating the same argument for $\ell = 1,...,m-1$, we deduce that
$\cl_{r,m}(\xi)$ is a period with respect to $F^{d-r+1}H_v^{2d-2r+m+1}(X,\C)$,
which implies that  $\ul{\cl}_{r,m}(\xi) = 0$.
\end{proof}

\section{Derivatives of normal functions II}\label{SEC022}

Consider the setting  in \S \ref{SEC01}, where $\lambda : \X \to \Ss$ is a smooth and 
proper map of smooth quasi-projective varieties, and where $X$ is a very general member. 
In this section, we will further assume that $\Ss$ is affine.
Associated to the Gauss-Manin connection $\nabla$ and the algebraic vector fields
$H^0(\Ss,\Theta_{\Ss})$ is a $D$-module of differential operators.  
If $\w\in   H^0(\Ss,\CO_{\Ss}\otimes R^i\lambda_{\ast}\C)  =
H^0(\Ss,\R^i\lambda_{\ast}\Omega^{\bullet}_{\X/\Ss})$ is an algebraic form, 
one can consider the ideal of partial differential operators with coefficients
in $\C(\Ss)$ annihilating $\w$, which will always be non-zero 
using the finite dimensionality of cohomology of the fibers of $\lambda$ and the fact that $\nabla$ is algebraic.
This section addresses  with the following question. 

\begin{q} { If the transcendental regulator
associated to $\Phi_{r,m}(\xi)$ is non-trivial, is the associated normal function $\nu$
associated to $\xi$ detectable by  a Picard-Fuchs operator $P\in I_{\w}$, for some
$\w \in F^{d-r+m+1}H_v^{2d-2r+m+1}(X,\C)$; namely is $P\langle \nu,\w\rangle \ne 0$?}
\end{q}

The answer is a definitive yes in the setting of Theorem \ref{T28}, provided Assumption
\ref{A1} (below) holds.
Again, the answer to this question is strongest  (unconditional yes) in the case
of families of $K3$ surfaces considered in this paper, including product variants such
as in Corollary \ref{C27} and Theorem \ref{T53}.  
We need the following mild assumption:
 \begin{assum}\label{A1}
 For a fixed choice of $r$ and $m$ above,
 \[
 \big\{R_v^{2r-m-1}\lambda_{\ast}\C\big\} \bigcap
 \big\{ \CO_{\Ss}\otimes F^rR_v^{2r-m-1}\lambda_{\ast}\C\big\}
 = 0.
 \]
 \end{assum}
 As one would expect, this assumption automatically holds in the situation of Theorem \ref{T27},
 as well as for the situation of families of $K3$ surfaces in this paper, as well
 as in Corollary \ref{C27} and Theorem \ref{T53}.

\subsection{Picard-Fuchs equations associated to regulators}

Much of the ideas in this section are inspired by  \cite{Gr}. Since (again) $\nabla$ is algebraic,
everything reduces to a local  calculation over a polydisk $S\subset \Ss$, in the analytic topology.
Recall that $\Theta_S$ is the holomorphic tangent bundle of $S$. We can think of $\Theta_S$ as the sheaf of holomorphic
linear differential operators which naturally carries a ring structure $\D_S$. That is, it is given by
\begin{equation}\label{E906}
\D_S = \CO_S\left[
\frac{\partial}{\partial z_1},
\frac{\partial}{\partial z_2},
...,
\frac{\partial}{\partial z_n}\right]
\end{equation}
where $\CO_S = \C[[z_1, z_2, ..., z_n]]$. By identifying $\partial/\partial z_k$ with
$\nabla_{\partial/\partial z_k}$, $\D_S$ acts on $\CO_S\tensor R^q\pi_* \C$ via
\begin{equation}\label{E907}
(v_1 v_2 ... v_l) \omega = \nabla_{v_1} \nabla_{v_2} ...\nabla_{v_l} \omega
\end{equation}
for $v_1, v_2, ..., v_l\in H^0(\Theta_S)$, where we write $H^0(-)$ for $H^0(S,-)$.
For $\omega\in H^0(\CO_S\otimes R^q \pi_*\C)$, we let
$I_\omega$ be the {\it Picard-Fuchs ideal\/} annihilating $\omega$, i.e.,
the left-side ideal consisting of differential operators $P\in H^0(\D_S)$
satisfying $P\omega = 0$.

\medskip
As in \S \ref{SEC02} let us again for simplicity restrict to the situation of a family of $K3$ surfaces.
We fix a nonzero section $\omega\in H^0(K_{\X/S})$.
For all $u\in H^0(\Theta_S)$ and all Picard-Fuchs operators
$P\in H^0(\D_S)$ such that
$P (\nabla_u \omega) = 0$, i.e., $P\in I_{\nabla_u \omega}$, it is obvious that
\begin{equation}\label{E902}
(Pu) \omega = 0
\end{equation}
and hence
\begin{equation}\label{E901}
(Pu) \langle \gamma, \omega \rangle = 0
\end{equation}
for all $\gamma\in H^2(X, \C)$ (where  $H^2(X, \C)$ is identified with $ H^0(S,R^2\pi_{\ast}\C)$,
using $S$ a polydisk).
Again let $\xi\in \CH^2(\X/S, 1)$ be the be the result of an algebraic deformation 
of a cycle in the central fiber $X$ restricted to $\X/S$, and $\cl_{2,1}$ be the regulator map
\begin{equation}\label{E9033}
\cl_{2,1}: \CH^2(\X/S, 1) \to H^0\biggl(\frac{\CO_S\tensor R^2 \pi_* \C}{\CO_S\tensor F^2
R^2\pi_* \C + R^2\pi_* \Q(2)}\biggr).
\end{equation}
We let $\nu$ be a lift of $\cl_{2,1}(\xi)$ to $H^0(\CO_S\tensor R^2\pi_* \C)$.

For $P\in I_{\nabla_u \omega}$, $Pu\in I_\omega$ ``kills'' all the periods
$\langle \gamma, \omega \rangle$ for $\gamma\in H^2(X, \Q(2))$. Therefore,
$(Pu) \langle \nu, \omega \rangle$ is independent of the choice of the lifting of $\cl_{2,1}(\xi)$.
Obviously, we have
\begin{equation}\label{E800}
Pu\langle \nu, \omega\rangle = P\left(\langle \nabla_u \nu, \omega \rangle +
\langle \nu, \nabla_u \omega\rangle \right).
\end{equation}

Since  $\langle \nabla_u \nu, \omega\rangle = 0$,
 \eqref{E800} becomes
\begin{equation}\label{E801}
P \left(u\langle \nu, \omega\rangle - \langle \nu, \nabla_u \omega\rangle\right) = 0
\end{equation}
which is a system of differential equations satisfied by $\cl_{2,1}(\xi)$. We put this into the following
proposition.

\begin{prop}\label{P001}
Let $\X/S$, $\nu$ and $\omega$ be given as above. Then \eqref{E801} holds for all
$u\in H^0(\Theta_S)$ and $P\in I_{\nabla_u \omega}$. Or equivalently,
\begin{equation}\label{E8021}
u\langle \nu, \omega\rangle - \langle \nu, \nabla_u \omega\rangle = 
\langle \gamma, \nabla_u \omega \rangle
\end{equation}
for some $\gamma\in H^2(X, \C) = H^0(S,R^2\pi_*\C)$.
\end{prop}

Here we need to say something about \eqref{E8021}. 
Namely, we want to say that the solutions of $P y = 0$ for $P\in I_{\nabla_u\omega}$ are generated
by $\langle \gamma, \nabla_u \omega \rangle$ for all $\nabla \gamma = 0$.
Roughly, it follows from \cite{Gr}(1.28).
It is actually more elementary than that as a consequence of the following observation,
which is a generalization of the fact that a function with
vanishing derivative is constant.

\begin{lemma}\label{LEM001}
Let $E$ be a flat holomorphic vector bundle over the polydisk $S$ with flat connection
$\nabla$ and let $I_\eta$ be the Picard-Fuchs ideal associated to an $\eta\in H^0(E)$
defined as above. Then the solutions of the system of differential equations
$Py = 0$ for $P\in I_\eta$ are generated as a vector space over $\C$ by
$\langle \gamma, \eta \rangle$ for all $\gamma\in E^\vee$ with $\nabla \gamma = 0$,
where $E^\vee$ is the dual of $E$.
\end{lemma}

\subsection{Nontriviality of Picard-Fuchs operators} \label{SSNPFO}

Suppose that any Picard-Fuchs operator in $I_{\w}$ annihilates $\cl_{2,1}(\xi)(\omega)$.
Then $\langle \nu, \omega \rangle  = \langle \gamma, \omega \rangle$
for some $\gamma\in H^2(X, \C)$. It follows that
\begin{equation}\label{E904}
(Pu) \langle \nu, \omega \rangle = 0
\end{equation}
for all $u\in H^0(\Theta_S)$ and $P\in I_{\nabla_u \omega}$.
By \propref{P001}, we have
\begin{equation}\label{E905}
(Pu) \langle \nu, \omega \rangle = P\langle \nu, \nabla_u \omega \rangle = 0
\end{equation}
for all $u\in H^0(\Theta_S)$ and $P\in I_{\nabla_u \omega}$ and
\begin{equation}\label{E908}
\langle \nu, \nabla_u \omega \rangle = \langle \gamma, \nabla_u \omega \rangle
\end{equation}
for some $\gamma\in H^2(X, \C)$.
Equivalently, on $H^0(\CO_S\tensor R^2\pi_*\C)$ and again after identifying
$H^2(X,\C)$ with  $H^0(S,R^2\pi_*\C)$ (recall again $S$ is a polydisk), we have
\begin{equation}\label{E909}
\nu \in [\nabla_u \omega ]^\perp + H^2(X, \C).
\end{equation}
That is, $\cl_{2,1}(\xi)$ lies in the image of $[\nabla_u \omega ]^\perp + H^2(X,\C)$, which we simply
write as
\begin{equation}\label{E911}
\cl_{2,1}(\xi) \in [\nabla_u \omega ]^\perp + H^2(X, \C).
\end{equation}

Assume for the moment that the family $\pi: \X\to S$ is \emph{maximal}, i.e., the image of the Kodaira-Spencer map $\kappa$ is
$[c_1(L)]^\perp$ at each point $t\in S$. Then by \coref{C28}, the projections of $\nabla_u \omega$ to
$\CO_S\otimes R^1 \pi_* \Omega_{\X/S}$ generate the subbundle $[c_1(L)]^\perp$ as $u$ varies in $H^0(\Theta_S)$.
And since
\begin{equation}\label{E955}
\cl_{2,1}(\xi)\in \bigcap_{u\in H^0(\Theta_S)} ([\nabla_u \omega ]^\perp + H^2(X,\C)) \cap
([\omega]^\perp + H^2(X, \C))
\end{equation}
and $c_1(L)\in H^2(X,\C)$, we see that
\begin{equation}\label{E956}
\cl_{2,1}(\xi)\in H^2(X, \C),
\end{equation}
i.e. the normal function has zero infinitesimal invariant, and hence zero
topological invariant by \cite{Sa}.
Let us explain this more precisely. 
If we consider for the moment the general setting in \S\ref{SEC01} of a smooth and proper morphism
$\lambda : \X \to \Ss$ of smooth quasi-projective varieties, where say $\Ss$
is affine, with space of
normal functions $\Ext^1_{V\MHS}(\Q(0),R^{2r-m-1}\lambda_{\ast}\Q(r))$, then there
is a short exact sequence:
\[
0\to J\big(H^{2r-m-1}_f(X,\Q(r))\big) \to \Ext^1_{V\MHS}(\Q(0),R^{2r-m-1}\lambda_{\ast}\Q(r))
\]
\[
\xrightarrow{\delta} \hom_{\MHS}\big(\Q(0),H^1(\Ss,R^{2r-m-1}\lambda_{\ast}\Q(r))\big)
\to 0,
\]
which induces an injection
\[
\Ext^1_{V\MHS}(\Q(0),R_v^{2r-m-1}\lambda_{\ast}\Q(r)) \hookrightarrow
\hom_{\MHS}\big(\Q(0),H^1(\Ss,R_v^{2r-m-1}\lambda_{\ast}\Q(r))\big),
\]
together with an injection (using $\Ss$ affine, see \cite{Sa})
\[
\hom_{\MHS}\big(\Q(0),H^1(\Ss,R_v^{2r-m-1}\lambda_{\ast}\Q(r))\big)\hookrightarrow
\nabla\Gamma J,
\]
where
\[
\nabla\Gamma J :=
\frac{\ker \nabla : H^0(\Ss,\Omega^1_{\Ss}\otimes F^{r-1}R_v^{2r-m-1}\lambda_{\ast}\C)
\to H^0(\Ss,\Omega^2_{\Ss}\otimes F^{r-2}R_v^{2r-m-1}\lambda_{\ast}\C)}{\nabla
H^0(\Ss,\CO_{\Ss}\otimes F^rR_v^{2r-m-1}\lambda_{\ast}\C)}.
\]
For $\nu \in \Ext^1_{V\MHS}(\Q(0),R^{2r-m-1}\lambda_{\ast}\Q(r))$,
$\delta \nu$ gives the topological invariant of $\nu$. Next, consider the sheaf
\[
\nabla J := \frac{\ker \nabla : \Omega^1_{\Ss}\otimes F^{r-1}R_v^{2r-m-1}\lambda_{\ast}\C
\to \Omega^2_{\Ss}\otimes F^{r-2}R_v^{2r-m-1}\lambda_{\ast}\C}{\nabla
\big(\CO_{\Ss}\otimes F^rR_v^{2r-m-1}\lambda_{\ast}\C\big)},
\]
with corresponding $\Gamma \nabla J := H^0(\Ss,\nabla J)$. By definition
of a normal function, one has the Griffiths infinitesimal invariant
$\delta_G\nu \in \Gamma \nabla J$. Under Assumption \ref{A1}, the natural
map $\nabla\Gamma J \to \Gamma \nabla J$ is an isomorphism.  
Indeed, by Assumption \ref{A1}, this follows from the short exact sequence:
\[
0\to \CO_{\Ss}\otimes F^rR_v^{2r-m-1}\lambda_{\ast}\C \xrightarrow{\nabla}
\big(\Omega^1_{\Ss}\otimes F^{r-1}R_v^{2r-m-1}\lambda_{\ast}\C\big)_{\ker \nabla}
\to \nabla J \to 0.
\]
Now
back to the case of our family of $K3$ surfaces,  with $(d,r,m) = (2,2,1)$,
Assumption \ref{A1} automatically holds, and the normal function $\nu$ associated to $\cl_{2,1}(\xi)$,
satisfies $\delta_G\nu = 0$, hence $\delta\nu = 0$.
This  implies that the
normal function arises from the fixed part $J\big(H_f^2(X,\Q(2))\big)$. This
cannot happen since  the reduced regulator $\ul{\cl}_{2,1}(\xi)$ is non-trivial.
Finally,  we use the fact that $T_0(S) \simeq H^1_{\alg}(X,\Theta_X)$ together
with \cite{C-L1} to deduce the non-trivially of the Picard-Fuchs operator acting
on a normal function arising from the general subfamilies in
the latter statement in Theorem \ref{T27}. A similar story holds
in the setting of Corollary \ref{CO78}, (and as will
be clearer later, as well as in Corollary \ref{C27} and Theorem \ref{T53}).  
Quite generally, in the setting of Theorem \ref{T28},
let us assume that any Picard-Fuchs operator applied to $\Phi_{r,m}$ is zero.
Then from the surjection of
\[
H^1_{\alg}(X,\Theta_X)\otimes H_v^{d-r+m+1-\ell,d-r+\ell}(X) \to H_v^{d-r+m-\ell,d-r+\ell+1}(X),
\]
in the case $\ell=0$, we deduce as in (\ref{E955}) that
\[
\cl_{r,m}(\xi)\in \big[H^0\big(\CO_S\otimes F^{d-r+m}R^{2d-2r+m+1}\pi_{\ast}\C)\big)\big]^{\perp}
+ H^{2d-2r+m+1}(X,\C).
\]
By iterating the same argument for $\ell = 1,...,m-1$, we deduce that
\[
\cl_{r,m}(\xi)\in \big[H^0\big(\CO_S\otimes F^{d-r+1}R^{2d-2r+m+1}\pi_{\ast}\C)\big)\big]^{\perp}
+ H^{2d-2r+m+1}(X,\C),
\]
which implies that the associated normal function has zero infinitesimal invariant,
and thus $\ul{\cl}_{r,m}(\xi) = 0$,  which is not the case.

\section{Proof of Theorem \ref{T53}}\label{SEC03}

In this section we restrict to the case
where $X$ is a projective $K3$ surface. We recall the real regulator
\begin{equation}\label{E100}
r_{3,1}:  \CH^3(X\times X, 1) \xrightarrow{}
H^{2,2}(X\times X, \R(2)).
\end{equation}

The image of $r_{3,1}$ thus contains
\begin{equation}\label{E002}
r_{3,1}(\CH^1(X)\tensor \CH^2(X,1))\tensor \R = H^{1,1}(X, \Q(1)) \tensor H^{1,1}(X, \R(1))
\end{equation}
for $X$ general and it also contains the class $[\Delta_X]$ of the diagonal. So it is natural to look at
the reduced real regulator
\begin{equation}\label{E001}
\ul{r}_{3,1}: \CH^3(X\times X, 1) \xrightarrow{r_{3,1}} H^{2,2}(X\times X, \R)
\xrightarrow{\text{projection}} V_X
\end{equation}
where
\begin{equation}\label{E003}
\begin{split}
V_X &= H^{2,2}(X\times X,\R)\cap 
(H^{1,1}(X, \Q(1))\tensor H^{1,1}(X,\R(1)))^\perp\\
&\quad \cap (H^{1,1}(X,\R(1))\tensor H^{1,1}(X,\Q(1)))^\perp \cap [\Delta_X]^\perp.
\end{split}
\end{equation}
It was proven in \cite{C-L2} that
\begin{equation}\label{E334}
{\rm Im}(\underline{r}_{3,1})\tensor \R \ne 0.
\end{equation}
Of course, this implies that the indecomposables
\begin{equation}\label{E008}
\CH^3_{\text{ind}}(X\times X, 1)\tensor \Q \ne 0
\end{equation}
for a general projective $K3$ surface $X$ \cite[Corollary 1.3]{C-L2}.

\bigskip

Now let us look at the transcendental part of $\cl_{3,1}$:
\begin{equation}\label{E009}
 \Phi_{3,1} : \CH^3(X\times X,1) \to \frac{\big\{F^3\big(H_v^2(X,\C)\otimes H^2_v(X,\C)\big)
 \big\}^{\vee}}{H_4(X\times X,\Q(1))},
\end{equation}
where now $X$ is a very general $K3$ and $H_v^2(X,\C)$ is transcendental
cohomology.
Although one can follow the same argument in \cite{C-L2} to prove that $\Phi_{3,1}$ is
non-trivial by a degeneration argument,
there is an easier way to derive this from Theorem \ref{T28}.
The proof of Corollary \ref{C27} is a stepping stone to the proof of the
stronger Theorem \ref{T53}.

\subsection{Non-triviality of $\Phi_{3,1}$}
It is instructive to explain precisely how Theorem \ref{T28} leads Corollary \ref{C27}, viz.,
to the non-triviality of $\Phi_{3,1}$ for $Y := X\times X$, where $X$ is a very general 
projective $K3$ surface. In this case $Y$ takes the role of $X$ in Theorem \ref{T27},
with $(d,r,m,\ell) = (4,3,1,0)$, $H^1_{\alg}(Y,\Theta_Y)$ will be identified with $H_{\alg}^1(X,\Theta_X)
\simeq \C^{19}$, and $H_v^{2d-2r+m+1}(Y,\Q) = H^4_v(Y,\Q)$ will be replaced by
\[
 [\Delta_X]^{\perp}\cap \big\{H_v^2(X,\Q)\otimes H^2_v(X,\Q)\big\},
\]
where $[\Delta_X]$ is the diagonal class. The pairing in Theorem \ref{T28} amounts to
studying the properties of the pairing
\[
H^1(\Theta_X) \otimes H^{3,1}(X\times X) \to H^{2,2}(X\times X),
\]
which amounts to a Gauss-Manin  derivative calculation.
So let $\X/S$ be a smooth projective family of $K3$ surfaces over a polydisk
$S$ (arising from a universal family), 
$\Y = \X\times_S\X$, $X = \X_0$ be a very general fiber of $\X/S$,
$Y = X\times X$ and $\pi_\X$ be the projection $\X\to S$. 
Let $\nabla$ be the GM connection associated to $\X/S$
and let $\alpha\in H^1(\Theta_X)$ be a tangent vector of $S$ at $0$. 
For $\omega\in H^0((\pi_\X)_* \wedge^2 \Omega_{\X/S})$ and $\eta \in H^0(R^1 (\pi_\X)_* \Omega_{\X/S})$, i.e., for
$\omega\in H^{2,0}(X)$ and $\eta\in H^{1,1}(X)$ when restricted to $X$,
we claim that
\begin{equation}\label{E018}
\bigcap_{\alpha, \omega, \eta} \Big((\nabla_\alpha(\omega\tensor \eta))^\perp \cap
(\nabla_\alpha(\eta\tensor \omega))^\perp \Big) \cap [\Delta_X]^\perp = \{ 0 \}
\end{equation}
in $H^{2,2}(Y)$ and hence the condition on the cup product pairing
in Theorem \ref{T28} holds.  Note that
\begin{equation}\label{E019}
[\nabla_\alpha(\omega\tensor \eta)] = [\nabla_\alpha \omega]\tensor \eta
+ \omega \tensor [\nabla_\alpha \eta]
\end{equation}
where $[\nabla_\alpha(\omega\tensor \eta)]$, $[\nabla_\alpha \omega]$ and $[\nabla_\alpha \eta]$ are
the projections of $\nabla_\alpha(\omega\tensor \eta)$, $\nabla_\alpha \omega$ and $\nabla_\alpha \eta$
onto $H^{2,2}(Y)$, $H^{1,1}(X)$ and $H^{0,2}(X)$, respectively. We know that
\begin{equation}\label{E020}
[\nabla_\alpha \omega] = \langle \alpha, \omega\rangle
\text{ and }
[\nabla_\alpha \eta] = \langle \alpha, \eta\rangle
\end{equation}
where $\langle\bullet, \bullet\rangle$ is the pairing
\begin{equation}\label{E021}
H^1(\Theta_X) \tensor (H^{1,1}(X) \oplus H^{2,0}(X)) \xrightarrow{} H^{0,2}(X) \oplus H^{1,1}(X).
\end{equation}
We write $\langle \alpha, \omega\rangle = \delta_\alpha \omega$
and $\langle \alpha, \eta\rangle = \delta_\alpha \eta$. Then \eqref{E018} follows directly from the following
statement.

\begin{prop}\label{PROP001}
For every complex $K3$ surface $X$,
\begin{equation}\label{E022}
\bigcap_{\alpha, \omega, \eta} \Big((\delta_\alpha \omega\tensor \eta + 
\omega \tensor \delta_\alpha \eta)^\perp \cap
(\delta_\alpha\eta\tensor \omega + \eta \tensor \delta_\alpha \omega)^\perp \Big) \cap [\Delta_X]^\perp = \{ 0 \}
\end{equation}
in $H^{2,2}(X\times X, \C)$, where $\alpha\in H^1(\Theta_X)$, $\omega\in H^{2,0}(X)$ and $\eta\in H^{1,1}(X)$.
\end{prop}

\begin{proof}
Combining Proposition \ref{PROP950} with the fact that
\begin{equation}\label{E031}
\langle \delta_\alpha \omega, \eta \rangle + 
\langle \omega, \delta_\alpha \eta \rangle = 0,
\end{equation}
we obtain
\begin{equation}\label{E029}
\langle [\Delta_X], \delta_\alpha\omega\tensor \eta + \omega \tensor \delta_\alpha \eta\rangle
= 0
\end{equation}
and hence
\begin{equation}\label{E030}
\begin{split}
&\quad \Span \{ \delta_\alpha\omega\tensor \eta  + \omega \tensor \delta_\alpha \eta\}
\\
&= [\Delta_X]^\perp \cap (H^{1,1}(X)\tensor H^{1,1}(X) \oplus H^{2,0}(X)\tensor H^{0,2}(X)).
\end{split}
\end{equation}
Similarly,
\begin{equation}\label{E032}
\begin{split}
&\quad \Span \{ \delta_\alpha\eta\tensor \omega + \eta \tensor \delta_\alpha \omega \}
\\
&= [\Delta_X]^\perp \cap (H^{0,2}(X)\tensor H^{2,0}(X) \oplus H^{1,1}(X)\tensor H^{1,1}(X))
\end{split}
\end{equation}
and \eqref{E022} follows easily.
\end{proof}

Note that $H_f^2(X, \C) = H^{1,1}(X, \Q(1))\tensor \C$ and $\pi_1(\Ss)$ acts on $H_v^2(X, \C)$
irreducibly. It is then not hard to see that
\begin{equation}\label{E011}
H_f^4(Y, \C) \cap H_v^2(X, \C)\tensor H_v^2(X, \C) \cap [\Delta_X]^\perp = \{ 0 \}
\end{equation}
and hence
\begin{equation}\label{E010}
H_f^4(Y,\C)\subset V_X^\perp.
\end{equation}

Since $\underline{r}_{3,1}(\xi) \ne 0$, this shows that $\Phi_{3,1}$ is non-trivial.

\subsection{The purely transcendental regulator $\Psi_{3,1}$}
 
We now turn our attention to the proof
of Theorem \ref{T53}.
More explicitly, we fix a nonvanishing holomorphic 2-form $\omega\in H^{2,0}(X)$ and look at
\begin{equation}\label{E004}
\langle \cl_{3,1}(\xi), \omega\tensor \omega\rangle
\end{equation}
modulo the periods $\int_\gamma \omega\tensor \omega$ for $\gamma\in H_4(X\times X, \Q(1))$.
We claim $\Psi_{3,1}$ is non-trivial, or equivalently,
$\langle \cl_{3,1}(\xi), \omega\tensor \omega\rangle$ is not a period for some 
$\xi\in \CH^3(X\times X, 1)$. Here we go slightly beyond the range
of $\ell$ in Theorem \ref{T28}, namely we allow $\ell = -1,\ 0$. More specifically we
consider 
\begin{equation}\label{E666}
\begin{split}
H_{\alg}^1(Y,\Theta_Y) \to \hom\big(H^{4,0}(Y),H^{3,1}(Y)\big),\\
H_{\alg}^1(Y,\Theta_Y)^{\otimes 2} \to   \hom\big(H^{4,0}(Y),H^{2,2}(Y)\big),
\end{split}
\end{equation}
where again $Y = X\times X$ is a self product of a very general projective $K3$
surface $X$, and $H_{\alg}^1(Y,\Theta_Y)$ is identified with the first order deformation
space of a universal family of projective $K3$'s. Of course if the former map
in (\ref{E666}) were surjective, then the latter map could  be replaced by
\[
H_{\alg}^1(Y,\Theta_Y) \to \hom\big(H^{3,1}(Y),H^{2,2}(Y)\big).
\]
Let us assume for the moment that
both maps in (\ref{E666}) are surjective. Then by the same reasoning as in the previous section,
one could argue that $\Psi_{3,1}$ is non-trivial.  However by a dimension count, it is clear
that both maps in (\ref{E666}) are not surjective. We remedy this by passing to the symmetric product
$\hat{Y} = Y/\langle \sigma\rangle$, where $\langle \sigma\rangle$ is the
symmetric group of order $2$  acting on $Y=X\times X$. In fact, insead
of working directly on $\hat{Y}$, we will work with the equivariant cohomologies
$H^4(Y,\Q)^{\sigma}$, and $\CH^3(Y,1)^{\sigma}$. That is, they
consist of classes fixed under $\sigma$. Note that $H^4(Y,\Q)^{\sigma}$
is still a Hodge structure. 
With the same setup for $\Phi_{3,1}$ and
following the same argument by differentiating,
we consider the orthogonal complements
\[
 (\nabla_\alpha(\omega\tensor\omega))^\perp 
\text{ and }
(\nabla_\beta\nabla_\alpha(\omega\tensor\omega))^\perp,
\]
following the situation in (\ref{E666}).
In particular, we are interested in the subspace
\begin{equation}\label{E033}
\begin{split}
&\bigcap_{\alpha,\beta} (\delta_\alpha\delta_\beta \omega \tensor \omega + 
\delta_\alpha \omega \tensor \delta_\beta \omega + \delta_\beta \omega \tensor
\delta_\alpha \omega + \omega \tensor \delta_\alpha \delta_\beta \omega)^\perp\cap\\
& \bigcap_\alpha (\delta_\alpha \omega \tensor \omega
+ \omega \tensor \delta_\alpha \omega)^\perp \cap (\omega\tensor \omega)^\perp
\cap [\Delta_X]^\perp
\end{split}
\end{equation}
when restricted to $Y$. Note that
\begin{equation}\label{E025}
\langle \delta_\alpha \omega, \delta_\beta \omega\rangle
+ \langle \omega, \delta_\alpha \delta_\beta \omega\rangle =
\langle \delta_\alpha \omega, \delta_\beta \omega\rangle +
\langle \omega, \delta_\beta \delta_\alpha \omega\rangle = 0
\end{equation}
by \eqref{E031} and hence
\begin{equation}\label{E034}
\delta_\alpha\delta_\beta \omega \tensor \omega + 
\delta_\alpha \omega \tensor \delta_\beta \omega + \delta_\beta \omega \tensor
\delta_\alpha \omega + \omega \tensor \delta_\alpha \delta_\beta \omega \in [\Delta_X]^\perp
\end{equation}
for all $\alpha,\beta\in H^1(\Theta_X)$.
Similarly,
\begin{equation}\label{E035}
\delta_\alpha \omega \tensor \omega
+ \omega \tensor \delta_\alpha \omega \in [\Delta_X]^\perp
\end{equation}
for all $\alpha\in H^1(\Theta_X)$.
Although we do not need it, \eqref{E025} also implies that $\delta_\alpha \delta_\beta = \delta_\beta\delta_\alpha$
and hence the map
\begin{equation}\label{E026}
H^1(\Theta_X) \tensor H^1(\Theta_X) \xrightarrow{} \hom(H^{2,0}(X), H^{0,2}(X))
\end{equation}
induced by $H^1(\Theta_X) \tensor H^1(\Theta_X) \tensor H^{2,0}(X) \to H^{0,2}(X)$ is a symmetric nondegenerate pairing. Obviously,
\begin{equation}\label{E036}
\begin{split}
&\quad
\Span\{
\delta_\alpha\delta_\beta \omega \tensor \omega + 
\delta_\alpha \omega \tensor \delta_\beta \omega + \delta_\beta \omega \tensor
\delta_\alpha \omega + \omega \tensor \delta_\alpha \delta_\beta \omega\}\\
& =[\Delta_X]^\perp \cap H^{2,2}(Y)^\sigma
\end{split} 
\end{equation}
and
\begin{equation}\label{E037}
\Span\{\delta_\alpha \omega \tensor \omega
+ \omega \tensor \delta_\alpha \omega\} = [\Delta_X]^\perp \cap H^{3,1}(Y)^\sigma
\end{equation}
by \eqref{E034}, \eqref{E035} and the nondegeneracy of \eqref{E026}. Therefore,
\begin{equation}\label{E038}
\begin{split}
&\bigcap_{\alpha,\beta} (\delta_\alpha\delta_\beta \omega \tensor \omega + 
\delta_\alpha \omega \tensor \delta_\beta \omega + \delta_\beta \omega \tensor
\delta_\alpha \omega + \omega \tensor \delta_\alpha \delta_\beta \omega)^\perp\cap\\
& \bigcap_\alpha (\delta_\alpha \omega \tensor \omega
+ \omega \tensor \delta_\alpha \omega)^\perp \cap (\omega\tensor \omega)^\perp
\cap [\Delta_X]^\perp \cap H^4(Y, \C)^\sigma = \{ 0 \}.
\end{split}
\end{equation}
Thus,
in order to prove \thmref{T53}, we just have to find $\xi$ such that
$\underline{r}_{3,1}(\xi)\ne 0$ and $\cl_{3,1}(\xi)\in H^4(Y,\C)^\sigma$. The obvious way to do this
is to find an equivariant higher Chow class $\xi\in \CH^3(Y, 1)^\sigma$ with $\underline{r}_{3,1}(\xi)\ne 0$.
Namely, we need a slightly stronger statement than \eqref{E334}. That is,

\begin{thm}\label{THM002}
There exists $\xi\in \CH^3(X\times X, 1)^\sigma$ such that
$\underline{r}_{3,1}(\xi) \ne 0$
for a general projective $K3$ surface $X$.
\end{thm}

\begin{proof} This is a consequence of the explicit construction of the cycle
in \cite{C-L2}.
\end{proof}

\section{Intermezzo: Lattice polarized $K3$ surfaces, hypersurface normal forms, 
and modular parametrization}\label{SECImz}

At this point it is natural to ask how one might construct explicit families of $K3$ surfaces satisfying the conditions of \thmref{T27}, with enough ``internal structure'' to make it possible to construct explicit cycles with nontrivial $\Phi_{2,1}$.  In light of \S \ref{SEC022}, it would also be highly desirable to have a means of explicitly constructing the Picard-Fuchs operators for these families.  

Families of the sort required by \thmref{T27} with a fixed generic N\'{e}ron-Severi lattice are known as \emph{lattice polarized $K3$ surfaces} \cite{Dol}.  Let $X$ be an algebraic $K3$ surface over the field of complex numbers.  If $M$ is an even lattice of signature $(1, \ell - 1)$ (with $\ell > 0$), then an \emph{$M$-polarization} on $X$ is a primitive lattice embedding
$$ i \ : \ M \hookrightarrow {\rm NS}(X)$$
such that the image $i(M)$ contains a pseudo-ample class.  There is also a coarse moduli space ${\mathcal M}_M$  for equivalence classes
of pairs $(M,i)$, which satisfies a version of the global Torelli theorem.  Moreover, surjectivity of the period map holds for families which are maximum in the sense of \S \ref{SSNPFO}; any family whose image in ${\mathcal M}_M$ is surjective satisfies this condition.  

An elliptic $K3$ surface with section consists of a triple $(X, \phi, S)$ of a $K3$ surface $X$, an elliptic fibration $\phi \ : \ X \rightarrow \mathbb{P}^1$, and a smooth rational curve $S \subset X$ forming a section of $\phi$.  This ``internal structure'' of an elliptic fibration with section on a $K3$ surface $X$ is equivalent to a lattice polarization of $X$ by the even rank two hyperbolic lattice
$$
H \ := \ \left( \begin{array}{cc} 0 & 1 \\ 1 & 0 \end{array} \right) 
$$
(see \cite[Theorem 2.3]{C-D1} for details).  The moduli space ${\mathcal M}_H$ of $H$-polarized $K3$ surfaces has complex dimension 18, and the generic elliptic $K3$ surface with section has 24 singular fibers of Kodaira type $I_1$.  Instead of working with a very general member of this family, which will have Picard rank $\ell = 2$, one can enhance the lattice polarization by considering a higher rank lattice $M$, with $H$ as a sublattice.  For each distinct embedding of $H$ into $M$, up to automorphisms of the ambient lattice $M$, we find an elliptic surface structure with section on all $M$-polarized $K3$ surfaces.  There is a decomposition of the N\'{e}ron-Severi lattice 
$$ {\rm NS}(X) = H \oplus W_X \ ,$$
where $W_X$ is the negative definite sublattice of $NS(X)$ generated by classes associated to algebraic cycles orthogonal to both the elliptic fiber and the section.  The sublattice
$$ W_X^{root} := \{ r \in W_X \ | \ \langle r , r \rangle = -2 \} $$
is called the {\em ADE type} of the elliptic fibration with section, as it decomposes naturally into the sum of ADE type sublattices spanned by $c_1$ of the irreducible (rational) components of the singular fibers of the elliptic fibration (see \cite[Section 6]{C-D1}).

For the explicit computations in \S \ref{SEC04} and \S \ref{SECTrReg} we will make essential use of one particular elliptic fibration with section on a family of $K3$ surfaces polarized by the lattice $H \oplus E_8 \oplus E_8$. It is not, in fact, the ``standard'' fibration, which corresponds to $W_X = E_8 \oplus E_8$, but the ``alternate fibration'' for which $W_X = D_{16}^+$ (the other even negative definite rank 16 lattice).  Up to ambient lattice automorphisms, these are the only two distinct embeddings of the lattice $H$ into $H \oplus E_8 \oplus E_8$.  As a result, we know that these are the only two elliptic fibrations with section on a very general member of this family of $K3$ surfaces \cite{C-D2}.

\subsection{Normal forms and elliptic fibrations} 

The natural setting for \thmref{T27} is families of lattice-polarized $K3$ surfaces which cover their corresponding coarse moduli spaces.  In order to effectively compute, we first need to construct such maximal families of $K3$ surfaces.  

The most classical construction of $K3$ surfaces is as smooth quartic (anticanonical) hypersurfaces in $\mathbb{P}^3$.  A very general member of this family will have a $4$-polarization and Picard rank $\ell = 1$.  It is possible, however, to construct subfamilies of smooth quartics with natural polarization by lattices of much higher rank.  For example, consider the ``Fermat quartic pencil''
\begin{equation}\label{Fermat} X_t := \{ x^4 + y^4 + z^4 + w^4 + t \cdot x y z w = 0 \} \subset \mathbb{P}^3 \ . \end{equation} 
For generic $t \in \mathbb{P}^1$, the group $G := (\mathbb{Z}/4\mathbb{Z})^2$ acts on $X_t$ by 
$$ x \mapsto \lambda \cdot x \ , \ y \mapsto \mu \cdot y \ , \ z \mapsto \lambda^{-1} \mu^{-1} \cdot z , $$
where $\lambda$ and $\mu$ are fourth roots of unity.  

The induced action of this group on the cohomology of $X_t$ fixes the holomorphic two-form $\omega_t$ (i.e., it acts symplectically).  Nikulin's classification of symplectic actions on $K3$ surfaces then implies that there is a rank 18 negative definite sublattice in the N\'{e}ron-Severi group of $X_t$, which together with the (fixed) $4$-polarization class means that the Picard rank of $X_t$ is at least 19.  As the family is not isotrivial, the Picard rank is not generically equal to 20, and we conclude that the family $X_t, t \in \mathbb{P}^1$ satisfies the conditions of  \thmref{T27} with $\ell = 19$.  (See \cite{Wh} for a general set of tools to bound the Picard rank of pencils of hypersurfaces with a high degree of symmetry.)  This is an example of a \emph{normal form} for the corresponding class of lattice polarized $K3$ surfaces, in this case providing a natural generalization of the Hesse pencil normal form for cubic curves in $\mathbb{P}^2$.

There is another family $Y_t$ of $K3$ surfaces with $\ell = 19$ easily derivable from the $X_t$ in (\ref{Fermat}) by quotienting each $X_t$ by the group $G$ and simultaneously resolving the resulting singularities in the family.  The family $Y_t$, known as the ``quartic mirror family,'' has rank 19 lattice polarization by the lattice $M_2 := H \oplus E_8 \oplus E_8 \oplus \langle -4 \rangle$.  

Another way to construct families of $4$-polarized $K3$ surfaces with an enhanced lattice polarization is to consider \emph{singular} quartic hypersurfaces in $\mathbb{P}^3$.  By introducing ordinary double point singularities of ADE type, it is a simple matter to engineer (upon minimal resolution) $K3$ surfaces with large negative definite sublattices of ADE type in their N\'{e}ron-Severi groups.  One feature that both the smooth and singular quartic hypersurface constructions enjoy is that for each line lying on the surface there is a corresponding elliptic fibration structure, defined by taking the pencil of planes passing through the line and considering the excess intersection of each (a pencil of cubic curves).  In this way, suitably nice quartic normal forms readily admit the structure of elliptic fibrations with section corresponding to various embeddings of the hyperbolic lattice $H$ into their polarizing lattices.

Let us illustrate this with the key example for the constructions in \S\ref{SEC04} and \S\ref{SECTrReg}, the singular quartic normal form for $K3$ surfaces polarized by the lattice
$$ M := H \oplus E_8 \oplus E_8$$
\cite{C-D2}.  Let $(X,i)$ be an $M$-polarized $K3$ surface.  The there exists a triple $(a,b,d) \in \mathbb{C}^3$, with $d \neq 0$ such that $(X,i)$ is isomorphic to the minimal resolution of the quartic surface
$$Q_M(a,b,d): \ y^2 z w - 4 x^3 z + 3 a x z w^2 + b z w^3 - \frac{1}{2} (d z^2 w^2 + w^4) = 0 \ .$$
Two such quartics $Q_M(a_1, b_1, d_1)$ and $Q_M(a_2, b_2, d_2)$ determine via minimal resolution isomorphic $M$-polarized $K3$ surfaces if and only if
$$(a_2, b_2, d_2) = (\lambda^2 a_1, \lambda^3 b_1, \lambda^6 d_1)$$
for some $\lambda \in \mathbb{C}^*$.  Thus the coarse moduli space for $M$-polarized $K3$ surfaces is the open variety
$$\mathcal{M}_M = \{ [a,b,d] \in \mathbb{W}\mathbb{P}(2,3,6) \ | \ d \neq 0 \} $$
with fundamental invariants
$$\frac{a^3}{d} \ \mbox{and} \ \frac{b^2}{d} \ .$$

On the singular quartic hypersurface $Q_M(a,b,d) \subset \mathbb{P}^3$ there are two distinct lines
$$ \{ x = w = 0 \} \ \mbox{and} \ \{ z = w = 0 \} \ ,$$
and the points 
$$P_1 := [0,1,0,0] \ \mbox{and} \  P_2 := [0,0,1,0]$$ 
are rational double point singularities on $Q_M(a,b,d)$ of ADE types $A_{11}$ and $E_6$ respectively.
The standard fibration is induced by the projection to $[z,w]$, and the alternate fibration is induced by the projection to $[x,w]$.  Moreover, among the exceptional rational curves in the resolution of $P_1$ are sections of both elliptic fibrations on $X(a,b,d)$; among the exceptional rational curves in the resolution of $P_2$ is a second section of the alternate fibration on $X(a,b,d)$.

It is useful to note that both the quartic mirror normal form $Y_t$ for $M_2$-polarized $K3$ surfaces and the $M$-polarized normal form $X(a,b,d)$ admit natural reinterpretations as the generic anticanonical hypersurfaces in certain toric Fano varieties \cite{Dor1,Dor2,CDLW}.  In both cases we build the toric Fano variety from the normal fan of a reflexive polytope.  For the $M_2$-polarized case, the polytope is the convex hull of 
$$ \{ (1,0,0), (0,1,0), (0,0,1), (-1,-1,-1) \} \subset \mathbb{R}^3 \ , $$ 
polar to the Newton polytope for $\mathbb{P}^3$.  For the $M$-polarized case, the polytope is the convex hull of
$$ \{ (1,0,0), (0,1,0), (0,0,1), (-1,-4,-6) \} \ ,$$
polar to the Newton polytope for $\mathbb{W}\mathbb{P}(1,1,4,6)$.  What is more, the two elliptic fibrations with section on a very general $X(a,b,d)$ are themselves induced by ambient toric fibrations on the toric variety in which it sits as a hypersurface.  Combinatorially, these correspond to reflexive ``slices'' of the corresponding polytope, i.e., planes in $\mathbb{R}^3$ which slice the reflexive polytope in a reflexive polygon.

\subsection{Picard-Fuchs equations and modular parametrization} 

There is a reverse nesting of moduli spaces corresponding to embeddings of the polarizing lattices.  In the context of the families $Y_t$ and $X(a,b,d)$ above, the usual embedding
$$ H \oplus E_8 \oplus E_8 \hookrightarrow H \oplus E_8 \oplus E_8 \oplus \langle -4 \rangle $$
corresponds to an algebraic parametrization 
$$ a(t) = (t + 16) (t + 256) \ , \ b(t)  = (t - 512)  (t - 8)  (t + 64) \ , \ d(t) = 2^{12} \, 3^6 \, t^3 $$
of a genus zero \emph{modular curve}.  To see the connection with classical modular curves, and indeed the Hodge-theoretic evidence for the underlying geometry, it is instructive to consider the Picard-Fuchs systems annihilating periods on the $K3$ surfaces involved.

Let $f(t)$ denote a period of the holomorphic $2$-form on $X(a,b,d)$.  The Griffiths-Dwork method for producing Picard-Fuchs systems yields (in an affine chart, where we have set $a = 1$)
$$ \left( \frac{\partial^2}{\partial b^2}  - 4  d \frac{\partial^2}{\partial d^2}  - 4 \frac{\partial}{\partial d} \right) f(b,d)  = 0$$
and
$$ \left( (-1 + b^2 + d) \frac{\partial^2}{\partial b^2}  + 2 b \frac{\partial}{\partial b}  + 4 b d \frac{\partial^2}{\partial b \partial d}  + 2 d \frac{\partial}{\partial d} + \frac{5}{36} \right) f(b,d) = 0$$ 
\cite{CDLW}.  By reparametrizing in terms of variables $j_1$ and $j_2$ 
$$b^2 = \frac{(j_1 - 1)(j_2 - 1)}{j_1 j_2} \ , \ d = \frac{1}{j_1 j_2}$$
we find that the Picard-Fuchs system completely decouples as
$$ 72 j_1 \left( 2 (j_1 - 1) j_1 \frac{\partial^2}{\partial j_1^2}  + (2 j_1 - 1) \frac{\partial}{\partial j_1} \right) f(j_1, j_2) - 5 f(j_1, j_2) = 0$$
and
$$ 72 j_2 \left( 2 (j_2 - 1) j_2 \frac{\partial^2}{\partial j_2^2}  + (2 j_2 - 1) \frac{\partial}{\partial j_2} \right) f(j_1, j_2) - 5 f(j_1, j_2) = 0 \ .$$
This implies that the periods of the $M$-polarized $K3$ surfaces split naturally as products $f(j_1, j_2) = f_1(j_1) \cdot f_2(j_2)$.

At this point it is natural to ask whether the second order ordinary differential equation satisfied by $f(j)$ is itself a Picard-Fuchs equation for a family of elliptic curves.  One can check for a family of elliptic curves over $\mathbb{P}^1_t$ in Weierstrass normal form
$$ \left\{ E_t \right\} := \left\{ y^2 z - 4 x^3 + g_2(t) x z^2 + g_3(t) z^3 = 0 \right\} \subset \mathbb{P}^2$$
that the periods of a suitably normalized holomorphic one-form on $E_t$
$$ g_2(t)^{\frac{1}{4}}  \frac{dx}{y} $$
satisfy Picard-Fuchs equations of the form of the second order equations above.  
Thus, by the Hodge Conjecture, we expect there to be an algebraic correspondence between $M$-polarized $K3$ surfaces and abelian surfaces (with principal polarization) which split as a product of a pair of elliptic curves.  This correspondence was made explicit in \cite{C-D2}; we recall the necessary features for our higher K-theory computations in \S \ref{SEC04} below.

What then is the meaning of the special subfamily $Y_t$ in terms of these split abelian surfaces?  When specialized to the subfamily $Y_t = X(a(t),b(t),c(t))$, the Griffiths-Dwork method produces the following Picard-Fuchs differential equation
$$ f^{(iii)}(t) + \frac{3 (3 t + 128)}{2 t (t + 64)}  f''(t)  +  \frac{13 t + 256}{4 t^2 (t + 64)} f'(t) + \frac{1}{8 t^2 (t+ 64)} f(t) = 0 \ .$$
On a general parametrized disk in the moduli space $\mathcal{M}_M$, the Picard-Fuchs ODE will have rank 4, just as the full Picard-Fuchs system.  The drop in rank indicates a special relationship between the two elliptic curves $E_{\tau_1}$ and $E_{\tau_2}$ corresponding to $Y_t$.  A differential algebraic characterization of the curves in $\mathcal{M}_M$ on which the Picard-Fuchs ODE drops in rank was given in \cite[Theorem 3.4]{CDLW}.  In fact, in the $M_2$-polarized case, the relationship is simply the existence of a two-isogeny between the two elliptic curves, i.e., $\tau_2 = 2 \cdot \tau_1$.  More generally, the $M_n$-polarized case corresponds to a cyclic $n$-isogeny, i.e., $\tau_2 = n \cdot \tau_1$.

Given that $M$-polarized $K3$ surfaces correspond to abelian surfaces which are the products of a pair of elliptic curves, the natural modular parameters on the (rational) coarse moduli space $\mathcal{M}_M$ are the elementary symmetric polynomials in the two $j$-invariants $j_1 = j(\tau_1)$ and $j_2 = j(\tau_2)$
$$\sigma := j_1 + j_2 \ \mbox{and} \ \pi := j_1 \cdot j_2 \ .$$
In this notation, it is easy to identify explicit rational curves in $\mathcal{M}_M$ over which the Picard-Fuchs differential equation has maximal rank ($=4$).  One such locus, which arises in the context of the construction of $K3$ surface fibered Calabi-Yau threefolds realizing hypergeometric variations, is specified by simply setting $\sigma = 1$ \cite{No}.  The Picard-Fuchs ODE has fourth order, and takes the following form
$$
f^{(iv)}(s) + \frac{2(4 s^2 - 3 s - 2)}{s(s - 1)(s + 1)} f^{(iii)}(s)  + \frac{1031 s^3 - 553 s^2 - 1175 s - 167}{72 s^2 (s - 1) (s + 1)^2}  f''(s)    $$
$$  + \,  \frac{167 s^2 - 239 s - 118}{36 s^2 (s - 1) (s + 1)^2}  f'(s)  +  \frac{385 (s - 1)^2}{20736 s^4 (s + 1)^2} f(s) = 0 $$
which splits as a tensor product of two very closely related factor second order ODEs
$$f''_1(s) + \frac{3 s + 1}{2 s (s + 1)} f'_1(s) + \frac{5}{144 s (s + 1)} f_1(s) = 0$$
and
$$f''_2(s) + \frac{3 s + 1}{2 s (s + 1)} f'_2(s) + \frac{5}{144 s^2 (s + 1)} f_2(s) = 0 $$
corresponding to the two families of elliptic curves satisfying $j_1(s) + j_2(s) = 1$.  Examples such as this provide a source of families of explicit non-maximal families of $K3$ surfaces to explore.

Instead of looking at superlattices of $H \oplus E_8 \oplus E_8$ such as $M_n$, one can consider sublattices such as $N := H \oplus E_7 \oplus E_8$ and $S := H \oplus E_7 \oplus E_7$ \cite{C-D3,C-D4}.  Moduli spaces of $K3$ surfaces polarized by these sublattices are themselves parametrized by modular functions (and contain $\mathcal{M}_M$ as a natural sublocus).   For example, there is a normal form for $N$-polarized $K3$ surfaces extending the singular quartic normal form for $M$-polarized $K3$ surfaces with one additional monomial deformation
$$Q_N(a,b,c,d): \ y^2 z w - 4 x^3 z + 3 a x z w^2 + b z w^3 + c x z^2 w - \frac{1}{2} (d z^2 w^2 + w^4) = 0 \ .$$
The associated coarse moduli space $\mathcal{M}_N$ is again an open subvariety of a weighted projective space
$$\mathcal{M}_N = \{ [a,b,c,d] \in \mathbb{W}\mathbb{P}(2,3,5,6) \ | \ c \neq 0 \ \mbox{or} \ d \neq 0 \} $$
with modular parametrization
$$ [a,b,c,d] = \left[ \mathcal{E}_4, \mathcal{E}_6, 2^{12} 3^5 \mathcal{C}_{10}, 2^{12} 3^6 \mathcal{C}_{12} \right] \ ,$$
where $\mathcal{E}_4$ and $\mathcal{E}_6$ are genus-two Eisenstein series of weights $4$ and $6$, and $\mathcal{C}_{10}$ and $\mathcal{C}_{12}$ are Igusa's cusp forms of weights $10$ and $12$ \cite[Theorem 1.5]{C-D3}.  

The connection to genus two curve moduli here is suggestive of the fundamental geometric 
fact that $N$-polarized $K3$ surfaces are Shioda-Inose surfaces coming from principally-polarized abelian surfaces.  The hypersurface normal form once again has two natural elliptic fibration structures with section, just as in the $M$-polarized case, and the Nikulin involution which gives rise to the Shioda-Inose structure can be seen most naturally as the operation of ``translation by $2$-torsion'' in the alternate elliptic fibration \cite{C-D4}.  There is a further extension to a 
normal form for $S$-polarized $K3$ surfaces.  In this case, most of the related geometric structures are still present, and we find a still more general modular parametrization of $\mathcal{M}_{S}$.  For all these families of lattice-polarized $K3$ surfaces in normal form, Picard-Fuchs equations can be obtained via the Griffiths-Dwork method applied directly to the singular quartic equations or in their realization as anticanonical hypersurfaces in Gorenstein toric Fano threefolds.  

The explicit computations which follow in \S\ref{SEC04} and \S\ref{SECTrReg} offer a glimpse of the range of phenomena surrounding \thmref{T27} which become accessible when we work with modular parametrizations of hypersurface normal forms for lattice polarized $K3$ surfaces equipped with well-chosen elliptic fibrations.  Both generalization to related higher-dimensional moduli spaces and manipulation of the associated explicit Picard-Fuchs systems now becomes possible.

\section{Explicit $K_{1}$ class on a family of Shioda-Inose $K3$ surfaces}\label{SEC04}

We now turn to a direct computation on the modular 2-parameter family
$\xab$ of $M:=H\oplus E_{8}\oplus E_{8}$-polarized (Picard-rank
18) $K3$'s introduced by Clingher and Doran {\cite{C-D2}}. Here $\xab$
($a,b\in\C$) is the minimal desingularization of \begin{equation}\left\{ Y^{2}Z-P(\theta)W^{2}Z-\frac{1}{2}Z^{2}W-\frac{1}{2}W^{3}=0\right\} \subset\PP_{[Y:Z:W]}^{2}\times\PP_{\theta}^{1},
\end{equation}where $P(\theta):=4\theta^{3}-3a\theta-b$. The results of {\cite{C-L1}}
already tell us that the real regulator map \begin{equation}r_{2,1}:\CH^{2}(\xab,1)\to\text{Hom}_{\R}(H_{v}^{1,1}(\xab,\mathbb{R}),\mathbb{R})
\end{equation}is generically surjective, making $\Phi_{2,1}$ nontrivial for very
general $(a,b)$. (We note that for those $\xab$ with Picard rank
18, $H_{v}^{1,1}=H_{tr}^{1,1}$.) The proof is based on \emph{non}-explicit
deformations of decomposable classes on Picard-rank 20 $K3$'s.

What we felt was missing here and in the literature are concrete indecomposable
cycles on which $r_{2,1}$ and $\Phi_{2,1}$ are nontrivial, particularly
those which arise naturally in the context of an internal elliptic
fibration. In our example, the projection $\xab\to\PP_{\theta}^{1}$
produces the so-called \emph{alternate fibration} with $6$ fibers
of Kodaira type $I_{1}$ and one fiber of type $I_{12}^{*}$. The
$I_{1}$ fibers provide \emph{the} most natural source of classes
in $\CH^{2}(\xab,1)$ provided one can show their real regulators are
nonzero.

This turns out to require some serious and interesting work, by first
passing to a \emph{Kummer} $K3$ family $\kab$ which is the minimal
resolution of both the quotient of $\xab$ by the Nikulin involution
and the quotient of a product of elliptic curves $E_{\alpha}\times E_{\beta}$
by $(-\text{id},-\text{id})$. This {}``intermediate'' setting seems
to be the one place where \emph{both} the normalization of the rational
curves supporting the family of $K_{1}$ classes (namely, a N\'eron
2-gon), and the closed $(1,1)$-form against which we integrate its
regulator current to compute $r_{2,1}$, are tractable. In fact, the
form has some singularities, even after pulling back the rational
curves, and so the computation requires careful additional justification.

\subsection{Kummer $K3$ geometry}

We begin with a review of special features of the Kummer family from
{\cite{C-D2}}, which has two parameters $\alpha,\beta\in\PP^{1}\backslash\{0,1,\infty\}$:
\begin{equation}\kabc':=\left\{ Z^{2}XY=(X-W)(X-\alpha W)(Y-W)(Y-\beta W)\}\right\} \subset\PP^{3}
\end{equation}is the singular model, with affine equation ($x,y,z=\frac{X}{W},\frac{Y}{W},\frac{Z}{W}$)
\begin{equation}z^{2}xy=(x-1)(x-\alpha)(y-1)(y-\beta),
\end{equation}and $\kab$ shall denote its minimal desingularization. Recall that
a Kummer is usually constructed by taking a pair of elliptic curves,
in this case \begin{equation}\xymatrix{
\left\{ u^2 = x(x-1)(x-\a) \right\} \ar @{=} [r] & : E_{\a} \ar @(ur,dr) []  & \mspace{-30mu} {\jmath_{\a}}:(x,u)\mapsto (x,-u)
\\
\left\{ v^2 = y(y-1)(y-\b) \right\} \ar @{=} [r] & : E_{\b} \ar @(ur,dr) []  & \mspace{-30mu} {\jmath_{\b}}:(y,v)\mapsto (y,-v) ,}
\end{equation}then taking the quotient $\kabc$ of $E_{\a}\times E_{\b}$ by the
automorphism $\jmath_{\a}\times\jmath_{\b}$. This is singular at
the image of the 16 products of 2-torsion points -- ordinary double
points whose resolution yields 16 exceptional $\PP^{1}$'s , and produces
$\kab$. 

In the following diagram of rational curves on $\kab$, the exceptional
divisors are represented by arcs; while the proper transforms of the
quotients of $E_{\a}\times\{\text{2-torsion point}\}$ resp. $\{\text{2-torsion point}\}\times E_{\b}$
are represented by horizontal resp. vertical lines:\begin{equation}\label{pic}\includegraphics[scale=0.6]{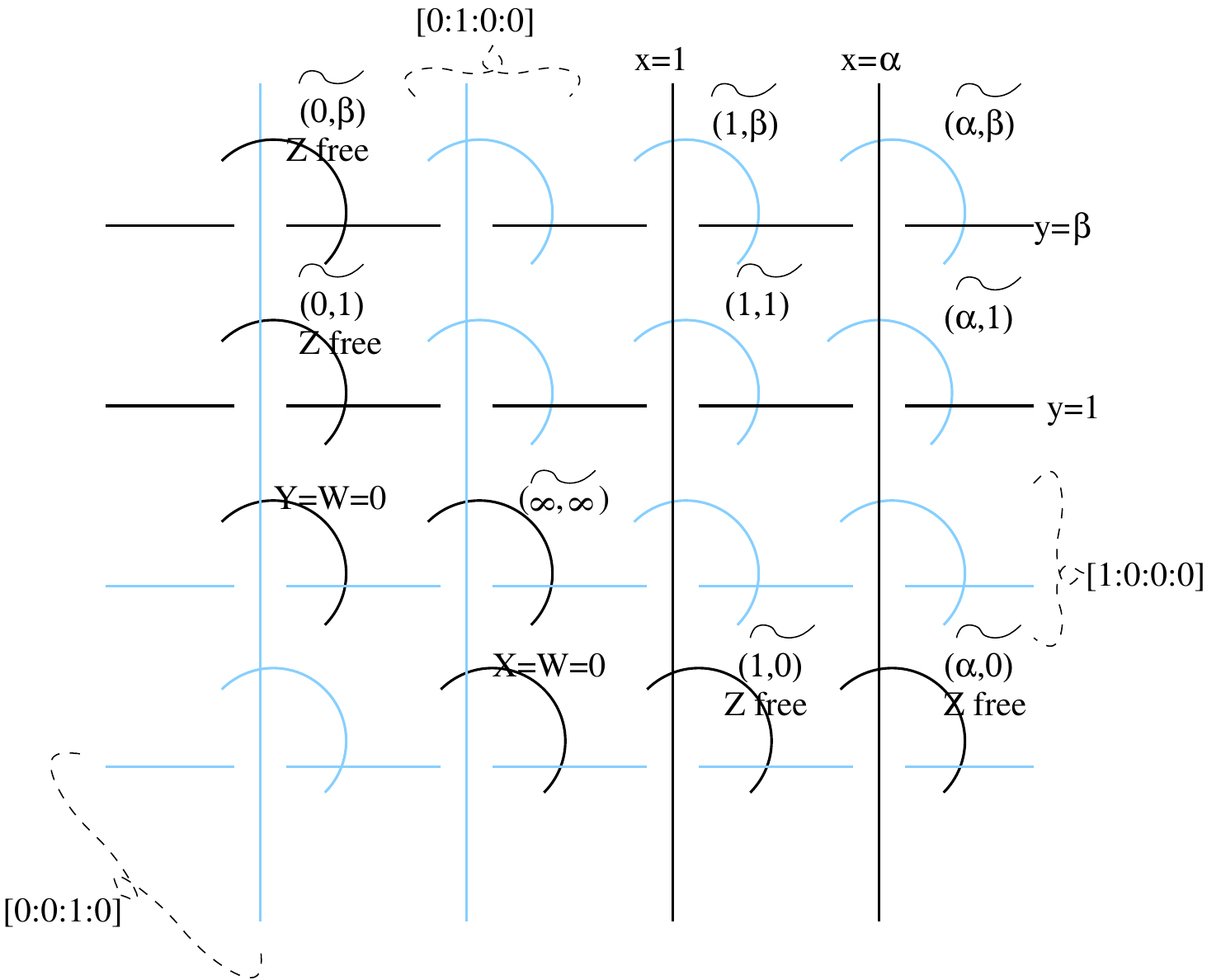}\end{equation} (Here
{}``$\widetilde{(\infty,\infty)}$'' stands for$\{W=0,\, XY=Z^{2}\}$.)
The projective model $\kabc'$ is the blow-down of $\kab$ along the
13 rational curves depicted more faintly. Notice that the configuration
\begin{equation}\includegraphics[scale=0.5]{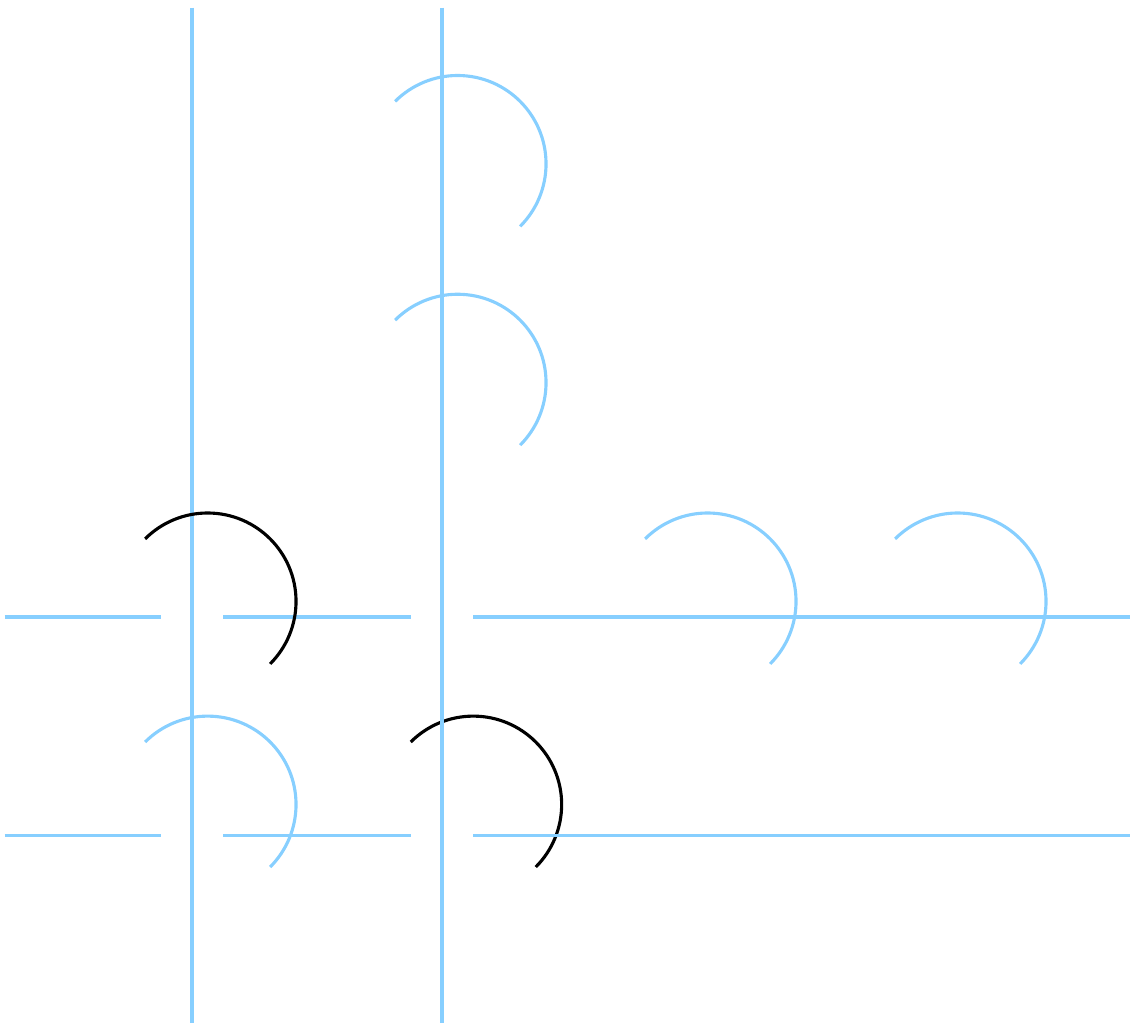}\end{equation}
has Dynkin diagram $D_{10}$, hence Kodaira type $I_{6}^{*}$.

We now describe an elliptic fibration of $\kab$ which shall have:
\begin{itemize}
\item this $I_{6}^{*}$ as its singular fiber at $\infty$;
\item the lines $y=1,\, y=\b,\, x=1,\, x=\a$ as sections;
\item the lines marked $\widetilde{(1,0)},\,\widetilde{(\a,0)},\,\widetilde{(0,1)},\,\widetilde{(0,\b)}$
as bi-sections;
\item the line marked $\widetilde{(\infty,\infty)}$ as a 4-section; and
\item 6 $I_{2}$ singular fibers, 4 of which have one of the lines marked
$\widetilde{(1,\b)}$, $\widetilde{(\a,\b)}$, $\widetilde{(1,1)}$,
or $\widetilde{(\a,1)}$ as one component.
\end{itemize}
Write \begin{equation}R(X,Y,W):=-\frac{X^{2}}{\a}-\frac{Y^{2}}{\b}+\frac{\a+1}{\a}XW+\frac{\b+1}{\b}YW-W^{2}.
\end{equation}Then the fibration, which is really nothing but the pencil $|I_{6}^{*}|$,
is given on the (singular) projective model by \begin{equation}\begin{array}{ccccc}
\kabc' & \longrightarrow & \PP^{1}\\
{}[X:Y:Z:W] & \longmapsto & [R(X,Y,W):XY] & =: & [\mu:1].\end{array}
\end{equation}In either case, the smooth elliptic fibers $\E_{\mu}$ (resp. $\Ec_{\mu}'$)
are double covers of the smooth conic curves \begin{equation}\c_{\mu}:=\left\{ R(X,Y,W)=\mu XY\right\} \subset\PP^{2},
\end{equation}branched over $(x,y)\;=\;\left(1,(1-\mu)\b+1\right),\,\left(\a,(1-\mu\a)\b+1\right),\,\left((1-\mu)\a+1,1\right),$
$\left((1-\mu\b)\a+1,\b\right).$ $\E_{\mu}$ is singular iff one
of the following hold:
\begin{itemize}
\item $\mu=\infty$: then $\E_{\infty}=I_{6}^{*}$;
\item $\mu\in\{1,\frac{1}{\a},\frac{1}{\b},\frac{1}{\a\b}\}$: then two
of the branch points collide, making $\Ec_{\mu}'$ into an $I_{1}$.
$\E_{\mu}$ is then the (Kodaira type $I_{2}$) union of its proper
transform with the exceptional divisor over the collision point --
for example, for $\mu=1$, $\E_{1}=\widetilde{\Ec_{1}'}\cup\widetilde{(1,1)}$;
or
\item $\mu\in\left\{ \frac{\a\b+1}{\a\b},\frac{\a+\b}{\a\b}\right\} $:
then the rational curve $\c_{\mu}$ acquires a node, so $\E_{\mu}$
has two nodes (again of type $I_{2}$).
\end{itemize}
This is all in case $J(E_{\alpha})\neq J(E_{\beta})$, i.e. $\b\notin\left\{ \a,\frac{1}{\a},1-\a,\frac{1}{1-\a},\frac{\a}{\a-1},\frac{\a-1}{\a}\right\} .$
Below we will eventually specialize to the case $\b=\a$, for which
generically $\E_{1}$ is still an $I_{2}$ but $\E_{\frac{1}{\a}=\frac{1}{\b}}$
becomes an $I_{4}$.

\subsection{Normalization of $\widetilde{\Ec_{1}'}$}

We will build our higher Chow cycle on $\E_{1}$. One can see right
away that it must have order-two monodromies about the components
of $(\PP^{1}\times\{0,1,\infty\})\cup(\{0,1,\infty\}\times\PP^{1})$,
since the tangent vectors of the $I_{1}$ fiber $\Ec_{1}'$ at its
singular point $(x,y,z)=(1,1,0)$ are $\left(1,-\frac{\b}{\a},\pm\sqrt{\frac{\b}{\a}(1-\a)(1-\b)}\right)$.
Notice that with $\a=\b$, the branches of the square root become
single-valued hence the monodromy will disappear; this will have consequences
later.

In order to compute, we need to parametrize $\Ec_{1}'$ by a $\PP^{1}$.
The first step is to do this for $\c_{1}$ using stereographic projection.
Putting $x=\Gamma+1$, $y=\xi\Gamma+1$ in its equation \begin{equation}\begin{matrix}
0\;=\;-\frac{x^{2}}{\a}-\frac{y^{2}}{\b}+\frac{\a+1}{\a}x+\frac{\b+1}{\b}y-1-xy \\
\\
=\;\cdots\;=-\left(\frac{1}{\a}+\frac{\xi^{2}}{\b}+\xi\right)\Gamma^{2}-\left(\frac{1}{\a}+\frac{\xi}{\b}\right)\Gamma
\end{matrix}
\end{equation} and solving for $\Gamma$, yields \begin{equation}\left(x(\xi),y(\xi)\right)\,=\,\left(\frac{\a\xi^{2}+\a(\b-1)\xi}{\Dx},\frac{\b(\a-1)\xi+\b}{\Dx}\right),
\end{equation} where $\Dx:=\a\xi^{2}+\a\b\xi+\b$.

The second step is to pull the affine equation of $\kabc'$ back along
$\xi\mapsto(x(\xi),y(\xi))$ and again use an analogue of stereographic
projection: \begin{equation}\begin{matrix}
z^{2}\;=\;\frac{(x-1)(x-\a)(y-1)(y-\b)}{xy}
\\ \\
=\;\cdots\;=\;\frac{(\a\xi+\b)^{2}(\xi+\b)(\a\xi+1)}{(\Dx)^{2}}.
\end{matrix}
\end{equation} So the equation of the $I_{1}$ fiber $\Ec_{1}'$ is \begin{equation}\label{i1eqn}(\Dx)^{2}z^{2}=(\xi+\b)(1+\a\xi)(\b+\a\xi)^{2},
\end{equation} which regarded as a curve in $\PP_{\xi}^{1}\times\PP_{z}^{1}$ has
bidegree $(4,2)$ and three nodes (hence of course genus $0$). A
curve of bidegree $(2,1)$ must meet $\Ec_{1}'$ in $8$ points with
multiplicity; so taking it to pass through the nodes $\left(-\frac{\b}{2}+\sqrt{\frac{\b^{2}}{4}-\frac{\b}{\a}},\infty\right)$,
$\left(-\frac{\b}{2}-\sqrt{\frac{\b^{2}}{4}-\frac{\b}{\a}},\infty\right)$,
$\left(-\frac{\b}{\a},0\right)$ and the smooth point $(-\b,0)$,
it must pass through one more point of $\Ec_{1}'$. Explicitly, these
curves are of the form \begin{equation}\label{curve}\Dx z=(\a\xi+\b)(\xi+\b)\gamma,
\end{equation} where $\gamma\in\C$ is a constant. To find the $\xi$-coordinate
of the residual point we square RHS(\ref{curve}) and set equal to
RHS(\ref{i1eqn}), which yields \begin{equation}\label{xigamma}\xi(\gamma)=\frac{1-\b\gamma^{2}}{\gamma^{2}-\a}.
\end{equation} Thinking of $\PP_{\gamma}^{1}$ as $\widetilde{\Ec_{1}'}$ and $\PP_{\xi}^{1}$
as $\c_{1}$, \eqref{xigamma} gives the branched double cover $\widetilde{\Ec_{1}'}\twoheadrightarrow\Ec_{1}'\twoheadrightarrow\c_{1}$,
where the first map just identifies a pair of points -- namely, those
with $\gamma^{2}=\delta:=\frac{\a\b-\a}{\b-\a\b}$. The following
table illustrates the relationship between functions on $\Ec_{1}'$:\begin{equation}\begin{tabular}{|c|c|c|}
\hline 
$\gamma^{2}$ & $\xi$ & $(x,y)$\tabularnewline
\hline
\hline 
$0$ & $-\nicefrac{1}{\alpha}$ & $(\alpha(1-\beta)+1,\beta)$\tabularnewline
\hline 
$\infty$ & $-\beta$ & $(\alpha,\beta(1-\alpha)+1)$\tabularnewline
\hline 
$\delta$ & $-\nicefrac{\beta}{\alpha}$ & $(1,1)$\tabularnewline
\hline 
$\nicefrac{1}{\beta}$ & $0$ & $(0,1)$\tabularnewline
\hline 
$\alpha$ & $\infty$ & $(1,0)$\tabularnewline
\hline 
$-\alpha\beta+\alpha+1$ & $1-\beta$ & $(0,\beta)$\tabularnewline
\hline 
$\frac{1}{1+\beta-\alpha\beta}$ & $\frac{1}{1-\alpha}$ & $(\alpha,0)$\tabularnewline
\hline 
roots of $\Delta(\xi(\gamma^{2}))$ & roots of $\Delta(\xi)$ & $(\infty,\infty)$\tabularnewline
\hline
\end{tabular}\end{equation} The rows starting with $0$ and $\infty$ correspond
to the branch points of $\Ec_{1}'\to\c_{1}$.

The third and last step is to find a coordinate $\z$ on $\widetilde{\Ec_{1}'}(\cong\PP^{1})$
which is $0$ and $\infty$ (rather than $\pm\sqrt{\delta}$) at the
two points mapping to the node of $\Ec_{1}'$, and $\pm1$ at the
two branch points of $\widetilde{\Ec_{1}'}\to\c_{1}$. This is given
by \begin{equation}\z=\frac{\gamma+\sqrt{\delta}}{\gamma-\sqrt{\delta}}\;\;\longleftrightarrow\;\;\gamma=\sqrt{\delta}\frac{\z+1}{\z-1}.
\end{equation} Our higher Chow cycle in $CH^{2}(\kab,1)$ will then simply be \begin{equation}\Zr_{\a,\b}:=\left(\widetilde{\Ec_{1}'},\z\right)+\left(\widetilde{(1,1)},g\right),
\end{equation} where $g$ has zero and pole cancelling with those of $\z$. (Note
that while $\z$ is the {}``preferred'' cordinate on the $\PP^{1}$,
we will work mainly in $\gamma$ below since this simplifies computations.)
We remark that $\Zr_{\a,\b}$ is defined as long as $\a,\b\notin\{0,1,\infty\}$
and $1\notin\left\{ \frac{1}{\a},\frac{1}{\b},\frac{1}{\a\b},\frac{\a\b+1}{\a\b},\frac{\a+\b}{\a\b}\right\} $,
but not quite well-defined: there is the issue of sign in $\z^{\pm1}$
(or equivalently, $\pm\sqrt{\delta}$) which leads to the predicted
order-2 monodromies.

\subsection{The $(1,1)$ current}

On $E_{\a}\times E_{\b}$ there is the closed, real-analytic $(1,1)$-form
\begin{equation}\label{omega}\omega=\frac{dx}{u}\wedge\overline{\left(\frac{dy}{v}\right)}=\frac{dx}{\sqrt{x(x-1)(x-\a)}}\wedge\overline{\left(\frac{dy}{\sqrt{y(y-1)(y-\b)}}\right)},
\end{equation} and $\w+\bar{\w}$, $i(\w-\bar{\w})$ obviously span $H_{tr,\R}^{1,1}$.
Clearly $\w$ is invariant under $\jmath_{\a}\times\jmath_{\b}$,
hence is the pullback of a $(1,1)$-current on $\kabc$, whose pullback\footnote{technically
these observations should be expressed in terms of push-forwards,
but the computations are better done as formal pullbacks.} $\w_{\k}$
to $\kab$ has integrable singularities along
the exceptional divisors: if locally the equation of one looks like
$w=0$, then there is a term of the form $\frac{dw\wedge d\bar{w}}{|w|}$.
Now we could argue that this current $\w_{\k}$ is closed and represents
a class in $H_{tr}^{1,1}(\kab,\C)$; but this approach runs into trouble
because $\widetilde{(1,1)}$, where part of the cycle is supported,
is an exceptional divisor. (The current's singularity along this divisor
makes the pairing {}``improper'', even though it {}``formally pulls
back'' to zero there.) Therefore, we will simply carry out an \emph{ad
hoc} pairing between\footnote{pairing the regulator with
$\w_{\k}+\overline{\w_{\k}}$ and $i(\w_{\k}-\overline{\w_{\k}})$
to get two real numbers, is equivalent to pairing it with $\w_{\k}$
to get a single complex number.
} $r_{2,1}(\Zr_{\a,\b})$ and $\w_{\k}$ on $\widetilde{\Ec_{1}'}$,
then interpret it on $E_{\a}\times E_{\b}$ where $\w$ is smooth.

So taking $\imath_{1}$ to denote the inclusion $\widetilde{\Ec_{1}'}\hookrightarrow\kab$,
we must compute $\imath_{1}^{*}\w_{\k}$. This is done by {}``formally''
pulling back the above form \eqref{omega} under $\xi\mapsto(x(\xi),y(\xi))$:
after some calculation, we obtain \begin{equation}\frac{-(\a\xi^{2}+2\b\xi+\b(\b-1))\overline{(\a(\a-1)\xi^{2}+2\a\xi+\b)}\, d\xi\wedge d\bar{\xi}}{|\Dx||\xi||\a\xi+\b||\xi+(\b-1)||(\a-1)\xi+1|\sqrt{(\xi+\b)\overline{(\a\xi+1)}}},
\end{equation} a sort of multivalued form on $\c_{1}$. Pulling this back (again
{}``formally'') to $\widetilde{\Ec_{1}'}\cong\PP_{\gamma}^{1}$
via $\gamma\mapsto\xi(\gamma)$ then yields (with apologies to the
reader) $\imath_{1}^{*}\w_{\k}=$ \begin{equation}\begin{matrix}
\frac{-4|\a\b-1|}{|\b||1-\a|}\cdot\frac{\left\{ (\a\b^{2}-\b^{2}-\b)\gamma^{4}+2\b\gamma^{2}+(\a^{2}\b^{2}-\a^{2}\b+\a-2\a\b)\right\} \,\gamma d\gamma}{|\gamma^{2}-\a||1-\b\gamma^{2}||\gamma^{2}-\delta||\gamma^{2}-(1+\a-\a\b)|}\wedge
\\ \\
\frac{\overline{\left\{ (\a^{2}\b^{2}-\a\b^{2}+\b-2\a\b)\gamma^{4}+2\a\gamma^{2}+(\a^{2}\b-\a^{2}-\a)\right\} }\, d\bar{\gamma}}{|(1+\b-\a\b)\gamma^{2}-1||\b\gamma^{4}+(\a^{2}\b^{2}-3\a\b)\gamma^{2}+\a|}.
\end{matrix}
\end{equation} While complicated, the 14 poles of this $(1,1)$ current are all
of the integrable form mentioned above, and their locations are precisely
the points where $\Ec_{1}'$ hits the exceptional divisors: $\widetilde{(1,1)}$,
$\widetilde{(1,0)}$, $\widetilde{(\a,0)}$, $\widetilde{(0,1)}$,
$\widetilde{(0,\b)}$ twice each; $\widetilde{(\infty,\infty)}$ four
times.

Along the locus $\a=\b$, this form simplifies a little: $\imath_{1}^{*}\w_{\k}=$
\begin{equation}\label{pullback}\begin{matrix}
-4|\a+1|\cdot\frac{\left\{ (\a^{2}-\a-1)\gamma^{4}+2\gamma^{2}+(\a^{3}-\a^{2}-2\a+1)\right\} \,\gamma d\gamma}{|\gamma^{2}-\a||1-\a\gamma^{2}||\gamma^{2}+1||\gamma^{2}-(1+\a-\a^{2})|}\wedge
\\ \\
\frac{\overline{\left\{ (\a^{3}-\a^{2}-2\a+1)\gamma^{4}+2\gamma^{2}+(\a^{2}-\a-1)\right\} }\, d\bar{\gamma}}{|(1+\a-\a^{2})\gamma^{2}-1||\gamma^{4}+(\a^{3}-3\a)\gamma^{2}+1|}.
\end{matrix}
\end{equation}

\subsection{The pairing}

The next step is simply to integrate $\log|\z|$ against $\imath_{1}^{*}\w_{\k}$
on $\widetilde{\Ec_{1}'}$. As $\log|\z|=\log\left|\frac{\gamma+\sqrt{\delta}}{\gamma-\sqrt{\delta}}\right|$,
this integral will have a multivalued behavior as indicated above.
It is singular but absolutely convergent: the worst behavior is at
$\gamma=\pm\sqrt{\delta}$ where it locally takes the form $\int_{D_{\epsilon}}\frac{\log|z|}{|z|}dz\wedge d\bar{z}$,
which is equivalent to $\int_{0}^{\epsilon}(\log r)dr$.

But setting $\a=\b$ ($\implies\delta=-1$) kills this monodromy,
allowing for a well-defined choice of $\Zr_{\a,\a}\in CH^{2}(\xab,1)$
over $\PP^{1}\backslash\{0,1,\infty,-1,2\}$ (see the end of $\S6.2$).
On a smooth compactification of the total space $\mathcal{X}\overset{\rho}{\to}\PP_{\a}^{1}$,
the {}``total cycle'' is easily seen to have residues (i.e. $\log|\z|$
blows up) along $\x_{-1,-1}\cup\x_{2,2}$ \emph{only} (cf. the proof of Theorem 3.7 in \cite{Ke}). By the localization sequence
for higher Chow groups, it can in fact be extended to all of $\rho^{-1}(\PP^{1}\backslash\{-1,2\})$.
Most importantly, eliminating the monodromy makes the integrals \begin{equation}\label{psi}\psi(\a)=\int_{\PP^{1}}\log\left|\frac{\gamma+i}{\gamma-i}\right|\Re(\imath_{1}^{*}\w_{\k})\;,\;\;\;\eta(\a)=\int_{\PP^{1}}\log\left|\frac{\gamma+i}{\gamma-i}\right|\Im(\imath_{1}^{*}\w_{\k})
\end{equation} \emph{real-analytic functions} of $\a\in\PP^{1}\backslash\{0,1,\infty,-1,2\}$.

Now on $E_{\a}\times E_{\a}$, by considering pullbacks to the diagonal,
one sees immediately that $i(\w-\bar{\w})$ is the algebraic class
whilst $\w+\bar{\w}$ is the transcendental one. Clearly the same
story holds on $\k_{\a,\a}$. So to check generic indecomposability
of $\Zr_{\a,\a}$ we need to demonstrate that $\psi(\a)$ (rather
than $\eta(\a)$) is generically nonzero.\footnote{In fact, a simple change of
coordinates to $\tilde{z}=\frac{1}{z}$ shows that $\eta(\alpha)$ is identically
zero.} Clearly it will suffice to show that $\lim_{\a\to1}\psi(\a)\neq0$.

Setting $\a=1$ in \eqref{pullback} yields \begin{equation}\begin{array}{ccccc}
\imath_{1}^{*}\w_{\k} & = & \frac{-8|\gamma^{2}-1|^{4}\,\gamma d\gamma\wedge d\bar{\gamma}}{|\gamma^{2}-1|^{6}|\gamma^{2}+1|}\\
 & = & \frac{-8\gamma d\gamma\wedge d\bar{\gamma}}{|\gamma^{2}-1|^2|\gamma^{2}+1|} & = & \frac{16r\{i\cos\theta-\sin\theta\}dx\wedge dy}{|\gamma^{2}-1|^2|\gamma^{2}+1|},\end{array}
\end{equation}  where $\gamma=x+iy=re^{i\theta}$. Because of the cancellations in
the second step, it requires some analysis to prove that $\int_{\PP^{1}}\log|\z|\Re(\imath_{1}^{*}\w_{\k})$
\emph{at} $\a=1$ actually computes the limit of $\psi$. This is
done in the appendix to this section, and so we have\inputencoding{latin1}{
}\inputencoding{latin9}\begin{equation}\label{limit}-\frac{1}{16}\lim_{\a\to1}\psi(\a)=\int_{\PP^{1}}\frac{\log\left|\frac{\gamma+i}{\gamma-i}\right|r\sin\theta}{|\gamma^{2}-1|^2|\gamma^{2}+1|}dx\wedge dy.
\end{equation} Now simply notice that
\begin{itemize}
\item the integral over $\PP^{1}$ in \eqref{limit} is double that over
the upper half plane, since $\log\left|\frac{\gamma+i}{\gamma-i}\right|$
and $\sin\theta$ are both odd in $\gamma$; and
\item the integrand is (where nonsingular) strictly positive on the upper
half plane.
\end{itemize}
We conclude that \eqref{limit} is a positive real number, finishing
this part of the argument.
\begin{rem}
{\rm It is more natural to normalize $\w_K$, and hence
$\psi$, by dividing out by $\big|\int_{E_{\alpha}}\frac{dx}{y}
\wedge \ol{\big(\frac{dx}{y}\big)}\big|$.
One can show -- either using formula \eqref{psi} or from general
principles to be explained in {\cite{Ke}} -- that this modified $\psi$ is asymptotic
to a constant times $\log|\a+1|$ (resp. $\log|\a-2|$) as $\a\to-1$ (resp. $2$), and goes to zero as $\a\to0,1,\infty$.
The first approach is indicated in the appendix.}
\end{rem}

\subsection{Interpretation of the integrals }

From the generic nontriviality of $\psi(\a)$, we know that \begin{equation}\label{integral}\int_{\widetilde{\Ec_{1}'}}(\log|\z|)\imath_{1}^{*}\w_{\k}
\end{equation}  is nonzero for generic $\a,\b$. We will show that this integral
has meaning as an invariant of $\Zr_{\a,\b}$ in roundabout fashion,
by first exhibiting it as an invariant of a related cycle on $E_{\a}\times E_{\b}$.

For generic $\mu$, the image $\Ec_{\mu}$ of $\widetilde{\Ec_{\mu}'}$
in $\kabc$ is a curve with intersection numbers as follows: \begin{equation}\includegraphics[scale=0.5]{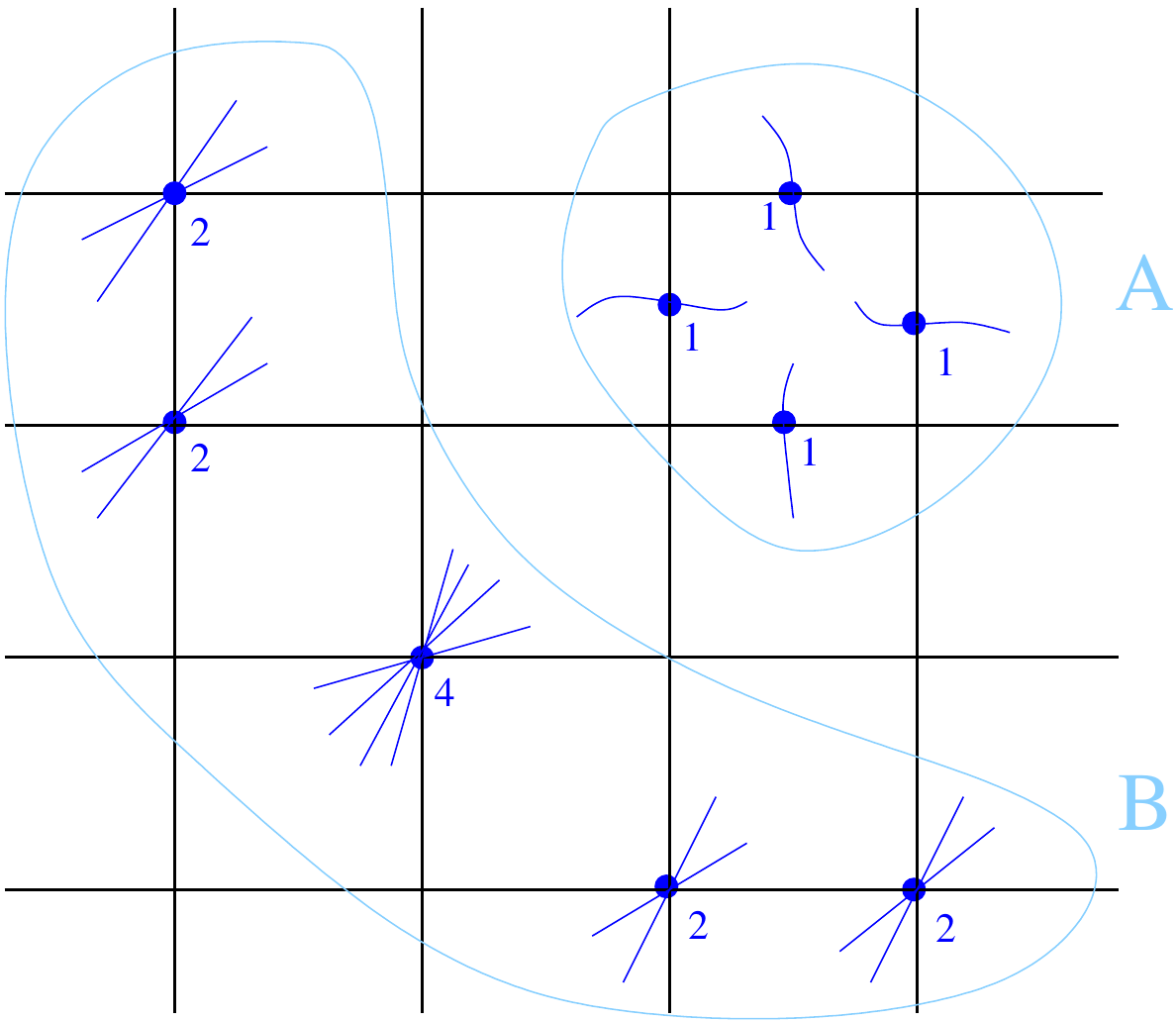}\end{equation}
where the horizontal and vertical lines have the same meaning as in
the earlier picture \eqref{pic}. Obviously its normalization is elliptic,
with 4 smooth branch points over the conic $\c_{\mu}$ at the points
of type $(\mathbf{A})$. Its preimage $\mathcal{D}_{\mu}$ in $E_{\a}\times E_{\b}$
is an irreducible curve with singularities at the points of type $(\mathbf{B})$;
and its normalization can be thought of as a double cover of the normalization
of $\Ec_{\mu}$, branched at the points lying over these singularities.
An easy Riemann-Hurwitz calculation shows that $\widetilde{\mathcal{D}_{\mu}}$
has genus $7$.

As $\mu\to1$, $\mathcal{D}_{\mu}$ and $\Ec_{\mu}$ each acquire
a new node, one mapping to the other: $\mathcal{O}\mapsto(1,1)$.
The local description (at the nodes) of the map $\mathcal{D}_{1}\twoheadrightarrow\Ec_{1}$
is {}``$z\mapsto z^{2}$'' on each branch separately. (Note that
$\widetilde{\mathcal{D}_{1}}$ has genus 6.) Therefore, the pullback
$\zt\in\C(\widetilde{\mathcal{D}_{1}})^{*}$ of the function $\z$
on $\widetilde{\Ec_{1}}$ pushes forward to $\mathcal{D}_{1}$ to
yield a $K_{1}$-class: its double-zero and double-pole cancel at
$\mathcal{O}$. Further, the real regulator current $\log|\zt|\delta_{\mathcal{D}_{1}}$
pairs against $\w\in\Gamma(E_{\a}\times E_{\b},{A}^{1,1})_{d-{\rm closed}}$ from \eqref{omega} to yield
\begin{lyxlist}{00.00.0000}
\item [{(a)}] an honest invariant of this $K_{1}$-class; and
\item [{(b)}] twice the value of the integral \eqref{integral}, since
$\w$ and $\zt$ are both invariant under the involution flipping
$\widetilde{\mathcal{D}_{1}}$ over $\widetilde{\Ec_{1}}$.
\end{lyxlist}
Consider the diagram \begin{equation}  \xymatrix{ & {\tilde{\k}_{\a,\b}} \ar @{->>} [ld]_{\pi_2} \ar @{->>} [rd]^{\pi_1} & & \tilde{\k}_{\a,\b}' \ar @{->>} [ld]_{\pi_1'} \ar @{->>} [rd]^{\pi_2'} \\ E_{\a} \times E_{\b} \ar @{->>} [rd]^{2:1} & & \kab \ar @{->>} [ld] \ar @{->>} [rd] \ar @{->>} [d] & & \xab \ar @{->>} [ld]^{2:1} \\ & \kabc & \kabc' & \kabc'' } 
\end{equation}in which $\xab$ is the Shioda-Inose $K3$, $\kabc''$ its quotient
by the Nikulin involution, and the relationship between the two sets
of parameters is given by \begin{equation}J(E_{\a})+J(E_{\b})=a^{3}-b^{2}+1\;,\;\;\; J(E_{\a})\cdot J(E_{\b})=a^{3}.
\end{equation}  The preimage of $\mathcal{D}_{1}$ under $\pi_{2}$ consists of $\widetilde{\mathcal{D}_{1}}$
and $\mathcal{W}$ (an exceptional $\PP^{1}$ with coordinate {}``$w$'')
meeting at $w=0$ and $w=\infty$ on $\mathcal{W}$. The map $\pi_{1}$
pushes this down to $\E_{1}=\widetilde{\Ec_{1}'}\cup\widetilde{(1,1)}$,
where the map from $\mathcal{W}$ to $\widetilde{(1,1)}$ is given
by $w\mapsto w^{2}$. Setting \begin{equation}\tilde{\Zr}_{\a,\b}:=(\widetilde{\mathcal{D}_{1}},\zt)+(\mathcal{W},w^{2})\in CH^{2}(\tilde{\k}_{\a,\b},1),
\end{equation}  we have $\pi_{1,*}(\tilde{\Zr}_{\a,\b})=2\Zr_{\a,\b}$ and $\pi_{2,*}(\tilde{\Zr}_{\a,\b})=\pi_{2,*}(\widetilde{\mathcal{D}_{1}},\zt)$.
By (a), (b), and functoriality of $r_{2,1}$, it now follows that
the pairing \eqref{integral} indeed computes the regulator of $\Zr_{\a,\b}$.

What about cycles on $\xab$? The 2:1 birational correspondence provided
by $\pi_{1}'$ and $\pi_{2}'$ identify its alternate fibration with
the elliptic fibration of $\kab$ (generically in 2:1 \'etale fashion).
More precisely, we have a diagram\inputencoding{latin1}{ }\inputencoding{latin9}\begin{equation}
\xymatrix{
\kab \ar [d] & \xab \ar @{-->} [l]_{2:1} \ar [d]
\\
\PP^1_{\mu} & \PP^1_{\theta} \ar [l]
} 
\end{equation} where the bottom map is of the form $\theta\mapsto q\theta+p$ (with
$p$ and $q$ constants dependent on $\a,\b$). On $\tilde{\k}_{\a,\b}'$
there is a $K_{1}$-cycle $\tilde{\Zr}_{\a,\b}'$ supported on an
$I_{2}$, $\pi_{1,*}'$ of which is $2\Zr_{\a,\b}$, and a similar
analysis goes through, proving indecomposability of the $I_{1}$-supported
$\mathcal{Z}_{a,b}:=\pi_{2,*}'(\tilde{\Zr}_{\a,\b}')$. So our integral
\eqref{integral} computes the real regulator of a trio of explicit
cycles, on $E_{\a}\times E_{\b}$, $\kab$, and $\xab$.

\subsection*{Appendix to Section 6}

Here we perform the analytic estimate which establishes the limiting
assertion in $\S6.4$, for $\a\to1$. It will suffice to consider
the behavior of the integral in a fixed neighborhood of one of the
points (we use $\gamma=+1$) where zeroes and poles collide. Write
$\chi=\a-1$, $\gamma^{2}=\zeta+1$, and let $D_{r}(c)$ denote the
open disk about $c$ of radius $r$. 

We may leave out the polynomial factors with no zero or pole approaching
$\zeta=0$, and approximate the locations of zeroes and poles to the
lowest order required to distinguish them. The problem is then to
show that \begin{equation}\label{est} \int_{|\zeta|<\frac{1}{2}}\frac{(\zeta-3\chi)\overline{(\zeta+3\chi)}(\zeta+\chi)\overline{(\zeta-\chi)}\log|\z|d\zeta\wedge d\bar{\zeta}}{|\zeta-(\chi+\chi^{2})||\zeta-(\chi-\chi^{2})||\zeta+(\chi+\chi^{2})||\zeta+(\chi-\chi^{2})||\zeta-i\sqrt{3}\chi||\zeta+i\sqrt{3}\chi|}
\end{equation}  limits to \begin{equation}\int_{|\zeta|<\frac{1}{2}}\log|\z|\frac{d\zeta\wedge d\bar{\zeta}}{|\zeta|^{2}}
\end{equation} as $\chi\to0^{+}$ along the real axis. Given $\epsilon>0$, and taking
$0<\chi<\epsilon/3$, it is obvious that the integrand in \eqref{est}
converges uniformly on $\epsilon<|\zeta|<1/2$. We claim that the
remaining part $\int_{|\zeta|<\epsilon}$ of the integral, independently
of $\chi\in(0,\frac{\epsilon}{3})$, is bounded by $1000\pi\epsilon$.
This will prove the desired convergence.

To verify the claim, we first remark that $\log|\z|$ is zero for
all $\zeta\in\PP^{1}(\R)$; in fact, we shall just use that $|\log|\z||<|\zeta|$.
Next, note that on the complement in $D_{\epsilon}(0)$ of the four
disks $D_{\frac{\chi}{2}}(\chi)$, $D_{\frac{\chi}{2}}(-\chi)$, $D_{\frac{\chi}{2}}(i\sqrt{3}\chi)$,
$D_{\frac{\chi}{2}}(-i\sqrt{3}\chi)$, \begin{equation}\frac{|\zeta+3\chi||\zeta+\chi|}{|\zeta+\chi+\chi^{2}||\zeta+\chi-\chi^{2}|}=\frac{|\lambda+2\chi||\lambda|}{|\lambda+\chi^{2}||\lambda-\chi^{2}|}=\frac{|1+\frac{2\chi}{\lambda}|}{|1+\frac{\chi^{2}}{\lambda}||1-\frac{\chi^{2}}{\lambda}|}
\end{equation} (where $\lambda:=\zeta+\chi$) is bounded by $6$, since $|\frac{2\chi}{\lambda}|\leq4$,
$|\frac{\chi^{2}}{\lambda}|\leq2\chi$ and we are assuming $\chi$
is small. The same is true for $\frac{|\zeta-3\chi||\zeta-\chi|}{|\zeta-\chi+\chi^{2}||\zeta-\chi-\chi^{2}|}$;
and similarly, $\frac{|\zeta|^{2}}{|\zeta-i\sqrt{3}\chi||\zeta+i\sqrt{3}\chi|}$
is bounded by $9$. So the integral over $D_{\epsilon}(0)\backslash\{\text{4 disks}\}$
is bounded by \begin{equation}\int_{|\zeta|<\epsilon}9\cdot6^{2}\cdot\frac{|d\zeta\wedge d\bar{\zeta}|}{|\zeta|}=324\cdot2\pi\int_{0}^{\epsilon}\frac{rdr}{r}<650\pi\epsilon.
\end{equation} Now consider (say) the right half of $D_{\frac{\chi}{2}}(-\chi)$:
here the absolute value of the integrand, apart from the $\frac{1}{|\lambda-\chi^{2}|}$,
is \begin{equation}\frac{|\lambda-4\chi||\lambda-2\chi|}{|\lambda-(2\chi+\chi^{2})||\lambda-(2\chi-\chi^{2})|}\cdot\frac{|\lambda|}{|\lambda+\chi^{2}|}\cdot\frac{|\lambda+2\chi||\lambda-\chi|}{|\lambda-i\sqrt{3}\chi||\lambda+i\sqrt{3}\chi|}\leq6\cdot1\cdot\frac{10}{3}\leq20.
\end{equation}  We have then \begin{equation}20\int_{D_{\frac{\chi}{2}}(0)\cap\Re(\lambda)>0}\frac{|d\lambda\wedge d\bar{\lambda}|}{|\lambda-\chi^{2}|}\leq20\int_{D_{\chi}(0)}\frac{|d\lambda\wedge d\bar{\lambda}|}{|\lambda|}=40\pi\chi<\frac{40}{3}\pi\epsilon,
\end{equation}  together with similar estimates on 3 other half-disks. The estimates
for $D_{\frac{\chi}{2}}(\pm i\sqrt{3}\chi)$ are each $\frac{250}{3}\pi\epsilon$.
Adding everything from inside and outside the 4 disks, we are safely
under $1000\pi\epsilon$.

We briefly address the situation at the other 4 points where poles in \eqref{pullback}
collide. The most striking case is that of $\a\to 2$. Substituting $\a=2$ in
$\int_{\PP^1}\log|\mathfrak{z}|\Re(\imath^*\w_{\k})$ yields the convergent integral
$$ -24\int_{\PP^1}\frac{\log\left|\frac{\gamma+i}{\gamma-i}\right| r \sin(\theta)}{|\gamma^2+1||\gamma^2-2||2\gamma^2-1|}dx\wedge dy. $$ Writing $\chi=\a-2$, $\gamma^2=\zeta-1$, to show this is $\lim_{\a\to2} \psi(\a)$
one must check (in analogy to \eqref{est}ff) that \begin{equation}\label{estat2} \int_{|\zeta|<\frac{1}{2}}
\frac{|\zeta+3i\sqrt{\chi}|^2|\zeta-3i\sqrt{\chi}|^2\log|\zeta|d\zeta\wedge d\bar{\zeta}}
{|\zeta||\zeta+3\chi||\zeta-3\chi||\zeta-3\sqrt{\chi}||\zeta+3\sqrt{\chi}|} \end{equation}
limits to $$ \int_{|\zeta|<\frac{1}{2}}\log|\zeta|\frac{d\zeta\wedge d\bar{\zeta}}{|\zeta|} $$
as $\chi\to 0$.
But this \emph{fails}, due to the rapid convergence to ($\zeta=$)$0$ of two of the poles; in fact,
\eqref{estat2} diverges logarithmically.

For $\a \to -1$, the limiting of the factor $|\a+1|\to 0$ in \eqref{pullback} is no match for the 
convergence of $7$ poles each to ($\gamma=$)$i$ and $-i$, again resulting in a logarithmic divergency for $\psi(\a)$. 
On the other hand, analyses
similar to (but simpler than) that for $\a\to1$ show $\lim_{\a\to 0}\psi(\a)$ and $\lim_{\a\to \infty}\psi(\a)$
to be convergent.

\section{The transcendental regulator for a Picard rank 20 $K3$}\label{SECTrReg}

Here we specialize to the case (cf. $\S6.5$) \begin{equation}\a=\frac{1}{2}=\b\;,\;\;\; a=1,\, b=0,
\end{equation} in which case $E_{\a},E_{\b}\cong\C/\Z\left\langle 1,i\right\rangle $
are CM and $p=3,\, q=-2$ (cf. {\cite{C-D2}}). The singular fibres are
at $\theta=$$\pm\frac{1}{2}$ (type $I_{2}$) and $\pm1$ (type $I_{1}$)
in $\x:=\x_{1,0}$, and at $\mu=$$2,4$ (type $I_{4}$) and $1,5$
(type $I_{2}$) in $\k_{\frac{1}{2},\frac{1}{2}}$. Recalling that
our original cycle was supported over $\mu=1$, which in this specialization
has \emph{remained} an $I_{2}$ fiber (hence preserving the cycle),
its transform $\mathcal{Z}:=\mathcal{Z}_{1,0}$ is supported over
$\theta=1$ in $\x$.

To take a closer look at the fibration structure of $\x$, we use
its affine equation \begin{equation}2y^{2}=w\underset{=:Q_{\theta}(w)}{\underbrace{(w^{2}+2\{4\theta^{3}-3\theta\}w+1)}}
\end{equation} to sketch the families of branch points of the elliptic fibers: \begin{equation}\includegraphics[scale=0.5]{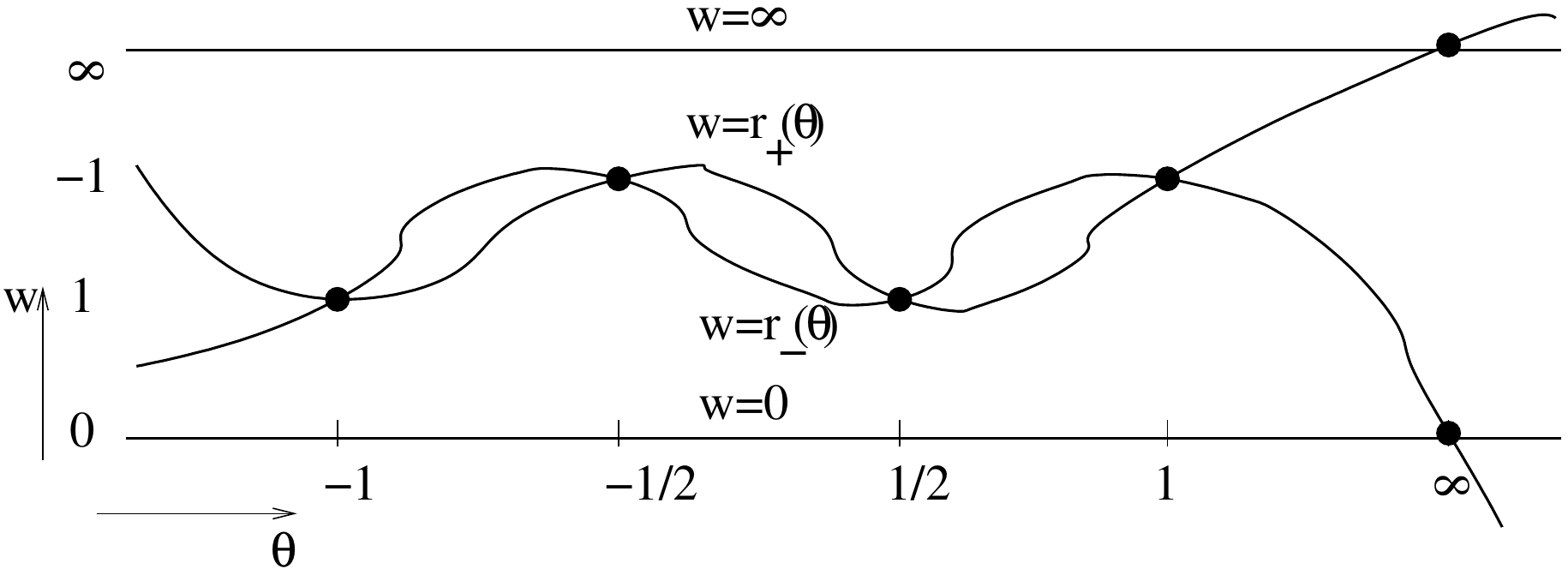}\end{equation}
Here $r_{\pm}(\theta)$ are the roots of $Q_{\theta}(w)$, which are
both negative real for $\theta\in[1,\infty)$, with $r_{-}=r_{+}^{-1}$.
For purposes of constructing transcendental cycles, one should imagine
all the branch points coealescing at $\theta=\infty$ since that fiber,
an $I_{12}^{*}$, has trivial $H_{1}$.

In particular, considering the fiber over $\theta=1$, the membrane
$\Gamma$ we use for the transcendental regulator computation must
bound on the indicated cycle $\del\Gamma=T_{\mathcal{Z}}$: \begin{equation}\includegraphics[scale=0.5]{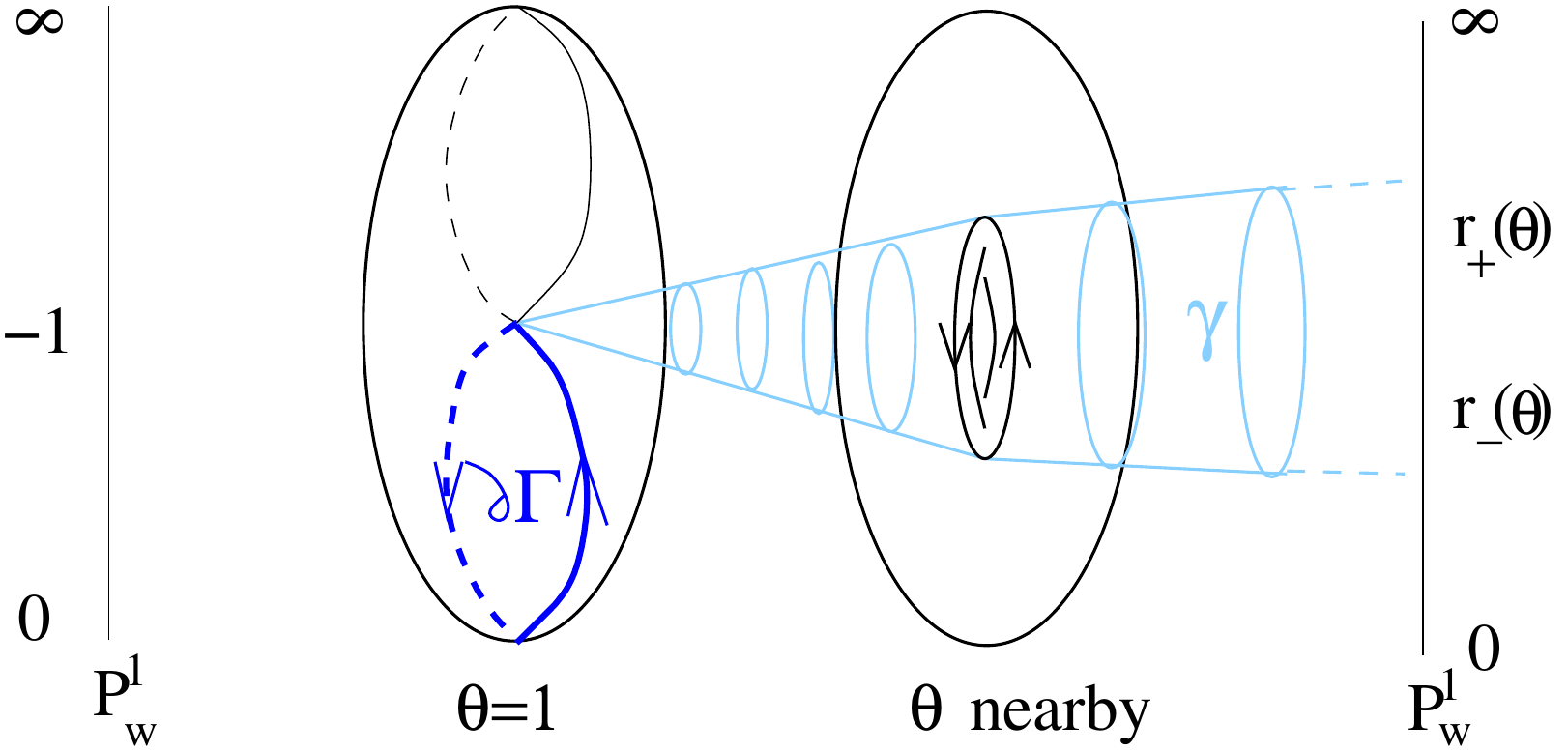}\end{equation}
which is a double cover of the path $[-1,0]\subset\PP_{w}^{1}$. The
transcendental 2-cycle $\gamma$ is the family of double covers of
$[r_{-},r_{+}]$ as $\theta$ goes from $1$ to $\infty$.

By basic residue theory the holomorphic $(2,0)$ form on $\x$ is
given by \begin{equation}\w_{0}=\frac{dw\wedge d\theta}{y}
\end{equation}in the affine coordinates. If \begin{equation}\int_{\gamma}\w_{0}=2\sqrt{2}\int_{\theta=1}^{\infty}\left(\int_{r_{-}(\theta)}^{r_{+}(\theta)}\frac{dw}{\sqrt{wQ_{\theta}(w)}}\right)d\theta\;(\,>0\,)
\end{equation} is one transcendental period, then using the automorphism $j:\x\to\x$
given by $(w,y,\theta)\mapsto(-w,-iy,-\theta)$, we have \begin{equation}\int_{j(\gamma)}\w_{0}=\int_{\gamma}j^{*}\w_{0}=i\int_{\gamma}\w_{0}.
\end{equation}  Normalizing $\w_{0}$ to $\w:=\frac{\w_{0}}{\int_{\gamma}i\w_{0}}$,
we find that $\Phi_{2,1}$ is described by \begin{equation}\begin{array}{ccc}
CH^{2}(\x,1) & \longrightarrow & \C/\Z[i]\\
\mathcal{Z} & \longmapsto & \int_{\Gamma}\w\end{array},
\end{equation}which for our particular cycle is \begin{equation}\label{kappa}\begin{matrix}
\kappa:=\int_{\Gamma}\w=2\int_{\theta=1}^{\infty}\int_{w=r_{+}(\theta)}^{0}\w
\\ \\
=\frac{\int_{1}^{\infty}\int_{r_{+}(\theta)}^{0}\frac{dw}{\sqrt{-wQ_{\theta}(w)}}d\theta}{\int_{1}^{\infty}\int_{r_{-}(\theta)}^{r_{+}(\theta)}\frac{dw}{\sqrt{wQ_{\theta}(w)}}d\theta}\in\R_{+}.
\end{matrix}
\end{equation} That is, \emph{the nontriviality of $\Phi_{2,1}(\mathcal{Z})_{\mathbb{Q}}$
is equivalent to irrationality of $\kappa$}. 

The situation is highly reminiscent of a computation by Harris {\cite{Ha}}
of the Abel-Jacobi map for the Ceresa cycle of the Fermat quartic
curve. In that case, a computer computation suggested that the comparable
invariant $\kappa'\in\R/\mathbb{Q}$ was nontrivial. This would have
implied that the cycle was nontorsion modulo rational equivalence,
a fact later proved by Bloch {\cite{B2}} using his $\ell$-adic $AJ$
map. Since the Fermat Jacobian is defined over $\bar{\mathbb{Q}}$,
the Bloch-Beilinson conjecture predicts injectivity of the usual $AJ$
map, and hence the irrationality of $\kappa'$. One might, in conclusion,
speculate that a similar story unfolds here.

\end{document}